\documentclass[11pt,a4paper]{amsart}
\usepackage[utf8]{inputenc}
\usepackage[english]{babel}
\usepackage{helvet, stmaryrd, ulem}
\usepackage{amsmath}
\usepackage{amsthm}
\usepackage{amsfonts}
\usepackage{amssymb}
\usepackage{mathtools}
\usepackage{mathrsfs}
\usepackage{dsfont}
\usepackage{tabularx}
\usepackage{xcolor}
\usepackage{bbm}
\usepackage{mhequ}
\usepackage{todonotes}
\usepackage[paper=a4paper,left=25mm,right=25mm,top=25mm,bottom=35mm]{geometry}
\usepackage{hyperref}
\usepackage[noabbrev,capitalize]{cleveref}
\usepackage{comment}

\hypersetup{pdftitle={Higher order approximation for nonlinear SPDEs}, pdfauthor={ Ana Djurdjevac, Máté Gerencsér, Helena Kremp}}

\newtheorem{theorem}{Theorem}[section]
     
\newtheorem{proposition}[theorem]{Proposition}
\newtheorem{lemma}[theorem]{Lemma}	

\newtheorem{corollary}[theorem]{Corollary}

\newtheorem{remark}[theorem]{Remark}
\newtheorem{assumption}[theorem]{Assumption}

\numberwithin{equation}{section}

\newcommand{\R}{\mathbb{R}}
\newcommand{\cR}{\mathcal{R}}

\newcommand{\T}{\mathbb{T}}
\newcommand{\N}{\mathbb{N}}
\newcommand{\E}{\mathds{E}}
\newcommand{\p}{\mathds{P}}

\newcommand{\F}{\mathcal{F}}
\newcommand{\cB}{\mathcal{B}}
\newcommand{\cM}{\mathcal{M}}

\newcommand{\calC}{\mathscr{C}}

\renewcommand{\mathcal}[1]{\mathscr{#1}}
\renewcommand{\epsilon}{\varepsilon}
\newcommand{\eps}{\varepsilon}
\newcommand{\bF}{\mathbb{F}}
\newcommand{\Z}{\mathbb{Z}}
\newcommand{\C}{\mathbb{C}}

\newcommand{\squeeze}[2][0]{\mbox{$\medmuskip=#1mu\displaystyle#2$}}

\newcommand{\Id}{\operatorname{Id}}

\DeclarePairedDelimiter{\abs}{\lvert}{\rvert}
\DeclarePairedDelimiter{\norm}{\lVert}{\rVert}

\DeclarePairedDelimiter{\paren}{(}{)}

\AtEndDocument{\bigskip{\footnotesize%
  \textsc{Freie Universität Berlin, Arnimallee 6, 14195 Berlin, Germany and University of Oxford, Mathematical Institute, Oxford, UK} \par  
  \textit{E-mail address}, A.~Djurdjevac: \texttt{ana.djurdjevac@maths.ox.ac.uk} \par
  \addvspace{\medskipamount}
  \textsc{TU Wien, Wiedner Hauptstra\ss e 8-10,
 1040 Vienna, Austria}\par
  \textit{E-mail address}, M.~Gerencsér: \texttt{mate.gerencser@asc.tuwien.ac.at} \par
  \addvspace{\medskipamount}
  \textsc{Weierstraß-Institut Berlin, Anton-Wilhelm-Amo-Straße 39, 10117 Berlin, Germany and Technische Universität Berlin, Straße des 17. Juni 136, 10587 Berlin, Germany} \par
  \textit{E-mail address}, H.~Kremp: \texttt{kremp@wias-berlin.de}
}}

\begin{document}

\title[Higher order approximation of nonlinear SPDEs]{Higher order approximation of nonlinear SPDEs with additive space-time white noise}

\date{\today}

\author{Ana Djurdjevac, Máté Gerencsér, Helena Kremp}


 

	\begin{abstract}
We consider strong approximations of $1+1$-dimensional stochastic PDEs driven by additive space-time white noise.
It has long been proposed \cite{Davie-Gaines, Jentzen-Kloden}, as well as observed in simulations, that approximation schemes based on samples from the stochastic convolution, rather than from increments of the underlying Wiener processes, should achieve significantly higher convergence rates with respect to the temporal timestep.
The present paper proves this.
For a large class of nonlinearities, with possibly superlinear growth, a temporal rate of (almost) $1$ is proven, a major improvement on the rate $1/4$ that is known to be optimal for schemes based on Wiener increments.
The spatial rate remains (almost) $1/2$ as it is standard in the literature.
		
		\bigskip
		
		 {{\sc Mathematics Subject Classification (2020):  60H15, 60H35 }
		
		}

	\end{abstract}

\maketitle

\section{Introduction}
We consider stochastic reaction-diffusion equations of the form
\begin{align}\label{eq:SPDE}
    \partial_{t}u=\Delta u+ f(u)+\xi, 
\end{align}
with $(t,x)\in\R^{+}\times\mathbb{T}$, space-time white noise $\xi$, a given nonlinearity $f:\R\to\R$ and initial condition $u_0$.
A common way to express the noise is to write $\xi$ as the (distributional) derivative $\partial_t\partial_x W$ of a $2$-dimensional Brownian sheet $W$.

Under some mild regularity assumptions on $f$,
existence and uniqueness of solutions to \eqref{eq:SPDE} is classical. 
In the present work we are interested in full discretisations of the equation. This question was first addressed in \cite{Gy}. A finite difference in space, (explicit or implicit) Euler method in time was studied based on sampling rectangular increments of $W$ 
on a grid with meshsize $M^{-1}$ in time and $N^{-1}$ in space, and $L^p$ rate of convergence of order $M^{-1/4}+N^{-1/2}$ was proved\footnote{To avoid obscuring the overview of the literature with $\eps$-s, for simplicity we do not distinguish between rate $\alpha$ and rate $\alpha-\eps$ for all $\eps>0$ in the mentioned results. }.
In \cite{Davie-Gaines} this rate is shown to be optimal in the sense that the conditional variance of the solution at a point given such samples of $W$ is lower bounded by a positive constant times $M^{-1/4}+N^{-1/2}$.
Similar upper and lower bounds are obtained in \cite{BGJK} for schemes based on a Galerkin truncation of $W$ in space and then sampling its increments in time. 

In an attempt to overcome the order barrier $1/4$ with respect to the temporal stepsize, \cite{Jentzen-Kloden} proposed a different scheme (already hinted at in \cite[Section~2.3]{Davie-Gaines}). The essential difference was the use of different functionals of the noise: instead of sampling increments of $W$, they used samples from the stochastic convolution with the semigroup generated by $\Delta$.
The stochastic convolution is a Gaussian process with explicitly known covariance, so that its sampling is straightforward.
\cite{Jentzen-Kloden} considered, instead of \eqref{eq:SPDE}, equations in a more abstract framework, viewing the nonlinearity  as a function $F:L^2(\T)\to L^2(\T)$ and the equation as a SDE on the Hilbert space $L^2(\T)$.
Under certain regularity conditions on $F$, the new scheme was shown to have a far superior rate of convergence $1$ with respect to the time stepsize, a major improvement on the previous rate $1/4$.
However, one of the main assumptions therein turns out to be too restrictive to allow for any(!) truly nonlinear SPDEs of the form \eqref{eq:SPDE}:
\begin{assumption}(\cite[part of Asn.~2.4]{Jentzen-Kloden})\label{asn:JK}
The map $F$ is Gateaux differentiable and there exists constant $L>0$ such that for all $u\in L^2(\T)$, $v\in\mathrm{Dom}(1-\Delta)$ one has
\begin{align}
        \big\|(1-\Delta)^{-1}F'(u)(1-\Delta)v\|_{L^2(\T)}\leq L\|v\|_{L^2(\T)}
\end{align}
\end{assumption}

\begin{proposition}\label{prop:JK-bad}
Let $f:\R\to\R$ be a differentiable function with bounded derivative and let $F:L^2(\T)\to L^2(\T)$ be defined as $F(u)(x)=f(u(x))$.
Then $F$ satisfies \cref{asn:JK} if and only if $f$ is affine linear.
\end{proposition}
\begin{remark}\label{rem:Frechet}
In \cite{Jentzen-Kloden} (and in a rather large portion of the literature) $F$ is in fact assumed to be not only Gateaux, but Fr\'echet differentiable.
Note that this already excludes all truly nonlinear Nemytskii operators:
If $F:L^2(\T)\to L^2(\T)$ is defined as $F(u)(x)=f(u(x))$ with $f$ being a differentiable function with bounded derivative, then there exists $u^{\star}\in L^{2}$ such that $F$ is Fr\'echet differentiable at $u^\star$ if and only if $f$ is affine linear (\cite[Proposition 2.8]{AP93}). 
\end{remark}
The proof is fairly straightforward and is given in Section \ref{sec:prelim}.
In light of Proposition \ref{prop:JK-bad}, the problem of ``overcoming of the order barrier'' for \eqref{eq:SPDE} remained open. Partial progress towards the conjectured rate $1$ has been made in \cite{Jentzen12} and \cite{Wang}, who proved rate $1/2$ in time for globally Lipschitz $f$ and cubic polynomial $f$ with negative leading order coefficient, respectively. Some partial results hinting at the possibility of a higher rate can be found in the theses \cite{thesis-1}, \cite{thesis-2}.
On the other side, temporal rate $1$ is proved in \cite{JKW, WQ, BrehierJH18, LT} partially overcoming the restrictive Assumption \ref{asn:JK}, but assuming a strong coloring condition on the noise instead, falling well short of the space-time white noise case.
In yet another direction, \cite{GS} proved rate $1/2$ in time (and even an improved rate $1$ in space) for \eqref{eq:SPDE} with polynomial $f$ with odd degree and negative leading order coefficient, even with just using Wiener increments, at the cost of measuring the error not in a space of functions, but rather a space of distributions. While such distributional norms are rather natural when dealing with higher dimensional stochastic reaction-diffusion equations (see e.g.~\cite{ma2021}), in the $1$ dimensional case it is desirable to bound the error in a genuine function space.

The aim of the present paper is to overcome all of the aforementioned caveats of \cite{Jentzen-Kloden, Jentzen12, JKW, Wang, GS}. We prove strong rate of convergence of rate $1-\eps$ in time for any $\eps$. The spatial error remains of order $N^{-1/2+\eps}$ as common in the preceding literature.
The methods can be applied in considerable generality, both in terms of the nonlinearity $f$ and the employed scheme. First, we consider $f$ that is globally bounded and has globally bounded derivatives. We show the aforementioned rate of convergence for a spectral Galerkin scheme in space and accelerated exponential explicit Euler in time. In this case the global bounds on $f$ allow simple a priori bounds as well as the application of Girsanov's theorem, which greatly simplifies the proof. The error estimate is uniform in space-time in this case.
Second, we consider $f$ that can grow polynomially, obeying the one-sided Lipschitz condition. For this class of equations, simple schemes like  exponential explicit Euler or standard Euler are not suitable due to the blow-up of such approximations, cf.~\cite{HutzenthalerSDE, beccari2019strong}. Instead one considers tamed schemes (see e.g. \cite{becker2023strong, BJ19, JP20, Wang}), splitting  schemes (see e.g.~\cite{BrehierJH18,BrehierG19}) or implicit Euler schemes (see e.g.~\cite{liu2019strong}), for which a priori bounds on the numerical solution can be derived (and therefore the blow-up is avoided). We employ a splitting scheme 
for the temporal discretisation and prove temporal rate $M^{-1+\epsilon}$
with the error measured uniform in time and $L^2$ in space.
The precise statements are formulated in Section \ref{sec:formulation} in Theorems \ref{thm:main} and \ref{thm:main2}.

The strategy leading to this improved rate relies on two key ingredients, inspired by \cite{GS} and \cite{BDG-SPDE}. The first important property to note is that the errors strongly depend on the topology in which they are measured. To illustrate this, consider the Ornstein-Uhlenbeck process $O$, that is, the solution to \eqref{eq:SPDE} with $f=0$, $u_0=0$. It is well-known that $O$ is almost $1/4$-H\"older continuous in time but no better. That is, for all $\eps>0$ there exist constants $c=c(\omega)>0$, $C=C(\eps,\omega)<\infty$ such that almost surely for all $0\leq s\leq t\leq 1$
\begin{align}\label{eq:O-triv}
   c|t-s|^{1/4}\leq \|O_t-O_s\|_{L^\infty(\T)}\leq C |t-s|^{1/4-\eps}.
\end{align}
However, moving to weaker topologies the time regularity of $O$ increases: for example, for any $\eps>0$ there exists $C=C(\eps,\omega)<\infty$ such that for all $0\leq s\leq t\leq 1$
\begin{align}\label{eq:O-lesstriv}
    \|O_t-O_s\|_{\calC^{-1/2}(\T)}\leq C |t-s|^{1/2-\eps},
\end{align}
where $\calC^{-1/2}(\T)$ is a Besov-H\"older space of negative regularity (see Section \ref{sec:prelim} below for details).
Since the regularity properties are naturally linked to rates of convergence of discretisations, one would like to leverage improved temporal regularity estimates like \eqref{eq:O-lesstriv} to obtain improved temporal rates.
This is the starting point to achieve $1/2$ temporal rate in a distributional norm \cite{GS}.
One key point of the present paper is to use similar ideas but still end up with error bounds in a functional norm.
Note however that this idea seems to stop at rate $1/2$: one can not weaken the topology further in \eqref{eq:O-lesstriv} to improve the temporal regularity. Indeed, even a single Fourier mode of $O$ is no better than $1/2$-H\"older continuous in time.

To illustrate the other main ingredient, consider the time integral
\begin{align}\label{eq:EM}
    E_M:=\Big|\int_0^T P_{T-s}\big(f(O_s)-f(O_{k_M(s)})\big)\,ds\Big|.
\end{align}
Here $P$ is the heat kernel, $M$ is the number of timesteps in an equidistant partition of the time horizon, and $k_M(s)$ is the last gridpoint before $s$.
In the error analysis, it turns out that $E_M$ determines the temporal rate.
However, efficient estimates for $E_M$ are highly nontrivial even if $f\in C_c^\infty$.
Using the triangle inequality, a global Lipschitz bound on $f$, and \eqref{eq:O-triv}, one easily obtains that $E_M\lesssim M^{-1/4+\epsilon}$, which is the classical error rate.
It is not clear how one could use \eqref{eq:O-lesstriv} in order to improve the rate, since a truly nonlinear function $f$ is not Lipschitz continuous with respect to the  $\calC^{-1/2}(\T)$-norm. Even if one could overcome this, the resulting error bound would only be $M^{-1/2+\epsilon}$.
The main improvement on estimating $E_M$ comes from \textit{not} using the triangle inequality.
Indeed, a fundamental idea coming from the field of regularisation by noise is that integrals along oscillatory processes enjoy a lot of cancellations, which are lost when bringing the absolute value inside the integral.
For example, in the case when $f$ is merely a bounded measurable function, a slight variation of \cite[Lemma~3.3.1]{BDG-SPDE} shows that $E_M\lesssim M^{-1/2 + \epsilon}$, where the triangle inequality would give no rate whatsoever.
A robust approach to obtain such improved estimates is the stochastic sewing strategy, originating from \cite{Khoa} and introduced to treat numerical analytic problems in \cite{BDG}. It is interesting to note however, that for many regularisation by noise arguments the $1+1$-dimensional stochastic heat equation behaves much like a fractional Brownian motion $B$ with Hurst parameter $1/4$, for which the best known rate for the analogue of $E_M$, namely the error
\begin{align}\label{eq:EM-SDE}
    \Big|\int_0^T f(B_s)-f(B_{k_M(s)})\,ds\Big|,
\end{align}
is $3/4-\eps$ (\cite[Lemma~4.1]{BDG}) even in the case of $f\in C^\infty_c$. 

Therefore, while neither of the two methods are sufficient on their own to obtain the desired rate, the aim of the present paper is to combine them in such a way that leverages the advantages of both the distributional power counting and the stochastic sewing, and yield the claimed temporal rate $1-\eps$. Since the expressions \eqref{eq:EM} and \eqref{eq:EM-SDE} differ by the semigroup $P$ inside the integral, the heuristic goal is to use it to improve the rate by lowering the spatial regularity where $f(O_s)-f(O_{k_M(s)})$ is estimated.

It is notable that our strong rate of $1$ for the temporal error even exceeds the best known weak error rates for splitting schemes of the Allen-Cahn equation, cf.~\cite{BG-weak}, or for exponential Euler schemes of nonlinear heat equations, cf.~\cite{Wang-weak}, where a temporal weak rate of $1/2$ is proven.

The stochastic Allen-Cahn equation with periodic boundary conditions being our prime example, we remark that periodic boundary conditions, although mathematically convenient in this paper, also have a physical relevance: in \cite{BG} the authors study Kramer's law, exhibiting different interesting behavior for periodic and Neumann boundary conditions. We expect that for our results accomodating different boundary conditions (Neumann or even Dirichlet as in \cite{FJL}) is doable without any essential difficulties, but needs more involved function spaces.

Unlike in the Wiener increment sample case \cite{Davie-Gaines}, for schemes based on sampling the stochastic convolution we are not aware of any existing lower error bounds. In \cref{prop:lb} below we thus give a short argument, which shows optimality of the temporal order $M^{-1}$. Moreover we provide numerical evidence in the case of the Allen-Cahn nonlinearity, which shows this temporal rate.

Let us end with a couple of remarks and open questions.
In the present paper we consider an accelerated exponential explicit Euler and a splitting scheme for the time discretization.
It would be interesting to extend the methods to other approximations, e.g. to implicit Euler or tamed schemes.
Furthermore, it seems promising to pursue this strategy for SPDEs whose nonlinearities do not only depend on the solution but also on its gradient, e.g. the stochastic Burgers' equation, whose nonlinearity is $\partial_x (u^2)$.

\section*{Acknowledgements}
The authors thank Arnulf Jentzen for bringing this interesting problem to their attention, providing many useful references, and pointing out the properties of Fr\'echet differentiability in Remark \ref{rem:Frechet}. The authors thank Owen Hearder for providing the code to produce the numerical plots.

ADj gratefully acknowledges funding by  Daimler and Benz Foundation as part of the scholarship program for junior professors and postdoctoral researchers. 
MG is funded by the European Union (ERC, SPDE, 101117125). Views and opinions expressed are however those of the author(s) only and do not necessarily reflect those of the European Union or the European Research Council Executive Agency. Neither the European Union nor the granting authority can be held responsible for them. The paper was written while HK was supported by the Austrian Science Fund (FWF) P34992. ADj and HK acknowledge support from the Deutsche Forschungsgemeinschaft (DFG) CRC/TRR 388 “Rough Analysis, Stochastic Dynamics and Related Fields” - Project ID 516748464.

\section{Set-up and Statement}\label{sec:formulation}
Let $(\Omega,\F,\p)$ be a probability space. We fix a time horizon $T>0$. The torus $\T$ is defined as $\T=\R/\Z$.
The space-time white noise $\xi$ is defined as a mapping from the Borel sets $\cB([0,T]\times\mathbb{T})$ into $L^{2}(\Omega)$, such that for any collection $A_{1},\dots,A_{k}\in \cB([0,T]\times\mathbb{T})$, the vector $(\xi(A_{1}),\dots,\xi(A_{k}))$ is Gaussian with zero mean and covariance $\E[\xi(A_{i})\xi(A_{j})]=\lambda(A_{i}\cap A_{j})$, where $\lambda$ denotes the Lebesgue measure.
We also fix a filtration $\bF=(\F_{t})_{t\in[0,T]}$, such that $(\Omega,\bF,\p)$ is complete and such that for any $t\in[0,T]$, $A\in \cB([0,t]\times\mathbb{T})$, $B\in \cB([t,T]\times\mathbb{T})$, $\xi(A)$ is $\F_{t}$-measurable and $\xi(B)$ is independent of $\F_{t}$. An example for $\bF$ would be the completed filtration generated by $\xi$.
The predictable $\sigma$-algebra on $\Omega\times[0,T]$ will be denoted by $\mathcal{P}$. Stochastic integrals $\int_{0}^{T}\int_{\mathbb{T}}g(s,y)\xi(ds,dy)$ against $\xi$ can be defined for all $\mathcal{P}\times \cB(\mathbb{T})$-measurable integrands $g:\Omega\times[0,T]\times\mathbb{T}\to\R$ with $g\in L^{2}(\Omega\times[0,T]\times\mathbb{T})$.
We refer to \cite{DPZ} for more details, but remark that for deterministic $f$ (which is the case used in the large majority of the article), the stochastic integral is simply the unique isometric and linear extension of the map $\mathbf{1}_A\mapsto\xi(A)$ to $L^2([0,T]\times\T)$.

For $k\in\Z$, denote the Fourier modes on $\T$ by $e_{k}(x)=e^{-2\pi ikx}$. The set $(e_k)_{k\in\Z}$ is an orthonormal basis of $L^2(\T,\C)$.
For $f\in L^1(\T,\C)$, its Fourier transform is denoted by $\mathcal{F} f (k)=\hat{f}(k)=\int_{\mathbb{T}}e^{-2\pi i k x}f(x)dx$, $k\in\Z$.
Denote by $(P_{t})_{t\geq 0}$ the heat semigroup on the torus, that is
$$P_{t}f=\mathcal{F}^{-1}(e^{-4\pi^2k^2t}\mathcal{F} f (k)).$$
One can equivalently write $P_t f=p_t\ast f=\int_{\R} p_{t}(\cdot-y)f(y)dy$, where
$$
p_{t}(x)=\sum_{k\in\mathbb{Z}}e^{-4\pi^2k^2t}e^{2\pi i k x}=\frac{1}{\sqrt{4\pi t}}\sum_{m\in\Z}e^{-(x-m)^2/4t}.
$$

We consider the mild formulation of \eqref{eq:SPDE}: 
\begin{align}\label{eq:mild}
    u_{t}=P_{t}u_{0}+\int_{0}^{t}P_{t-s}\big(f(u_{s})\big)ds+\int_{0}^{t}\int_{\mathbb{T}}p_{t-s}(x-y)\xi(ds,dy).
\end{align}
The stochastic convolution, also referred to as the Ornstein-Uhlenbeck (OU) process, is denoted by
\begin{align}\label{eq:OU}
    O_{t}:=\int_{0}^{t}\int_{\mathbb{T}}p_{t-s}(x-y)\xi(ds,dy), \quad t\geq 0.
\end{align}
The well-posedness of the mild formulation is classical under a one-sided Lipschitz and polynomial growth assumption on $f$, see Proposition \ref{prop:wp} below.

We start by setting up the result in the easier case of $f$ being bounded with bounded derivatives up to order $2$. In this setting, several steps of the proof are simplified.
Introducing the key ideas in this case hopefully benefits the reader in understanding the more general form. For now we work under the following assumption.
\begin{assumption}\label{asn:easier}
\begin{enumerate}
    \item[(a)] There exists a constant $K$ such that for all $i=0,1,2$ and all $x\in \R$ one has
    \begin{align}
        |\partial^i f(x)|\leq K,
    \end{align}
    with the convention that $\partial^0f=f$.
    \item[(b)] The initial condition $u_0$ is an $\F_0$-measurable random variable with values in $\calC^{1/2}(\T)$ and for any $p\in[1,\infty)$ there exists a constant $\cM(p)$ such that $\E\|u_{0}\|_{\calC^{1/2}}^p\leq\cM(p)$.
\end{enumerate} 
\end{assumption}

We use the approximation scheme exactly as in \cite{Jentzen-Kloden}, that is a spectral Galerkin scheme in space and accelerated exponential Euler scheme in time, defined as follows.
For $N\in\N$, let $\Pi_{N}$ denote the orthogonal projection from $L^{2}(\mathbb{T},\C)$ to the subspace $\text{span}(e_{k},\abs{k}\leq N)$.
Let $\Delta_{N}:=\Delta\Pi_{N}=\Pi_{N}\Delta$ and let $(P_{t}^{N}:=P_{t}\Pi_{N})_{t\geq 0}$ be the corresponding semigroup.
As before, one can write $P_t^N f=p_t^N\ast f$, where
$$
p_{t}^{N}(x)=\sum_{\abs{k}\leq N}e^{-4\pi^2k^2t}e^{2\pi i k x}.
$$
One then first defines a spatial approximation $U^{N}$ as the solution of the finite dimensional SDE
\begin{align}\label{eq:UN}
    U^{N}_{t}=P^{N}_{t}u_{0}+\int_{0}^{t}P_{t-s}^{N}f(U^{N}_{s})ds+O_{t}^{N},\quad t\in[0,T],
\end{align}
with the notation $O^N$ for the spatially discretised stochastic convolution
\begin{align}\label{eq:discrete-OU}
O^{N}_{t}=\int_{0}^{t}\int_{\mathbb{T}}p^{N}_{t-s}(x-y)\xi(ds,dy).
\end{align}

For $M\in\N$, let $h=T/M$ and consider the temporal gridpoints $t_{k}=kh$ for $k=0,\ldots,M$. Let for $s\in[0,1]$, $k_{M}(s):=\lfloor h^{-1}s\rfloor h$ be the last gridpoint before (or equal to) $s$.
The exponential Euler time discretisation   of \eqref{eq:UN}, that yields a space-time discretisation of \eqref{eq:mild},
is denoted by $V^{M,N}$ and defined inductively by setting
$V_{0}^{M,N}=\Pi_{N}u_{0}=:u_{0}^{N}$ and 
\begin{align*}
    V_{t_{k+1}}^{M,N}=P_{h}^{N}V_{t_{k}}^{M,N} + \Delta_{N}^{-1}(P_{h}^{N}-\operatorname{Id})\Pi_N(f(V_{t_k}^{M,N}))+O^{N}_{t_{k+1}}-P_{h}^{N}O_{t_{k}}^{N} 
\end{align*} for $k=0,\dots,M-1$.
Here $\operatorname{Id}$ denotes the identity operator on $L^{2}$ and we set by convention $\Delta_{N}^{-1}(P_{h}^{N}-\operatorname{Id})e_{0}:=h e_0$.
Alternatively (and often more conveniently) one can express $V^{M,N}$ also in a ``mild'' form
\begin{align}\label{eq:JK-scheme}
    V^{M,N}_{t_{k+1}}&=P_{t_{k+1}}^{N}u_{0}+\sum_{l=0}^{k}\int_{t_{l}}^{t_{l+1}}P_{t_{k+1}-s}^{N}f(V_{t_{l}}^{M,N})ds+O_{t_{k+1}}^{N}\nonumber
    \\&=P_{t_{k+1}}^{N}u_{0}+\int_{0}^{t_{k+1}}P_{t_{k+1}-s}^{N}f(V_{k_{M}(s)}^{M,N})ds+O_{t_{k+1}}^{N}, \quad k=0,\dots,M-1,
\end{align} 
 see e.g. \cite[Sec.~5.(b)]{Jentzen-Kloden}.
An advantage of the form \eqref{eq:JK-scheme} is that one can easily extend it to arbitrary (i.e. not grid-)points $t\in[0,T]$, simply by replacing each instance of $t_{k+1}$ on the right-hand side by $t$. We will frequently use this extension. 

\begin{theorem} \label{thm:main}
Let \cref{asn:easier} hold. 
    Let $u$ be the unique mild solution of \eqref{eq:SPDE} and for any $M,N\in\N$, let $V^{M,N}$ be as above.
    Then for any $\epsilon>0$ and $p\in[1,\infty)$ there exists a constant $C=C(T,\epsilon, p, K, \cM)$\footnote{Here and below, whenever $\theta$ is some collection of parameters, expressions of the form $C=C(\theta,\cM)$ mean that there exists a $p^*$ depending on $\theta$ such that the constant $C$ depends only on $\theta$ and $\cM(p^*)$.} such that 
   \begin{align}
      \big(\E\sup_{t\in[0,T]}\norm{u_{t}-V^{M,N}_{t}}_{L^{\infty}(\mathbb{T})}^p\big)^{1/p}\leq C \big(N^{-1/2+\epsilon}+M^{-1+\epsilon}\big).
   \end{align}
\end{theorem}

The proof of \cref{thm:main} is given in \cref{sec:bounded-f}.

In the second half of this article, we allow for superlinearily growing nonlinearities $f$ that satisfy a one-sided Lipschitz condition. A prime example is the Allen-Cahn nonlinearity $f(x)=x-x^3$. 
We now give the set up for the more general formulation of the article, which is substantially more technical. We start by the assumption on the nonlinearity.
\begin{assumption}\label{ass:f}
\begin{enumerate}
    \item[(a)] 
    There exists a constant $K\geq 1$ and $m\geq 0$ such that for any $i=0,1,2,3$ and all $x\in\R$ one has
\begin{align}\label{eq:f-ass1}
    |\partial^i f(x)|\leq K(1+|x|^{(2m+1-i)\vee 0}),
\end{align}
with the convention that $\partial^0f=f$, and furthermore
for all $x\in\R$ one has
\begin{align}\label{eq:f-ass2}
    \partial f(x)\leq K.
\end{align}
\item[(b)] The initial condition $u_0$ is an $\F_0$-measurable random variable with values in $\calC^{1/2}(\T)$ and for any $p\in[1,\infty)$ there exists a constant $\cM(p)$ such that $\E\|u_{0}\|_{\calC^{1/2}}^p\leq\cM(p)$.
\end{enumerate}
\end{assumption}
\begin{remark}
    \cref{ass:f} implies a local Lipschitz bound with polynomial growth and a global one-sided Lipschitz bound. That is, for all $x,y\in\R$ one has 
    \begin{align}
        |f(x)-f(y)|&\leq K(1+|x|^{2m}+|y|^{2m})|x-y|,
    \\
    (x-y)(f(x)-f(y))&\leq K |x-y|^2.
    \end{align}
\end{remark}

\begin{remark}
The assumption on the initial condition $u_0\in\calC^{1/2}$ in \cref{asn:easier} (b) and \cref{ass:f} (b) is minimal (up to $\epsilon$) in the sense that it guarantees a spatial rate $N^{-1/2+\epsilon}$ for any $\epsilon>0$ for the term related to the initial condition, while a more irregular initial condition would decrease the spatial rate. Indeed, we have that $\norm{u_0-\Pi_{N}u_{0}}_{L^{\infty}}\lesssim N^{-\alpha+\epsilon}\norm{u_0}_{\calC^{\alpha}}$ for any $\alpha\geq 0$ and $\epsilon>0$.
\end{remark}

For standard Euler or (accelerated) exponential Euler schemes for SPDEs with superlinearily growing coefficients, a priori estimates are known to fail (cf.~\cite{Hutzenthaler}). 
For our analysis, we consider the following splitting scheme: $X^{M,N}_{0}=\Pi_{N}u_{0}=u_{0}^{N}$ and
\begin{align}\label{eq:splitting}
    Y_{t_{k}}^{M,N}&=\Phi_{h}(X_{t_{k}}^{M,N})\\\nonumber
    X_{t_{k+1}}^{M,N}&=P_{h}^{N}Y_{t_{k}}^{M,N}+O^{N}_{t_{k+1}}-P_{h}^{N}O^{N}_{t_{k}}
\end{align}
for $k=0,\dots, M-1$ and $h=T/M$ and where $\Phi_t (z)$ solves  
\begin{align}\label{eq:Phi}
    \partial_t\Phi(z)=f(\Phi(z)),\quad \Phi_0(z)=z
\end{align} 
and $O^{N}$ is the truncated Ornstein-Uhlenbeck process defined in \eqref{eq:discrete-OU}.
\begin{remark}
Often the ODE \eqref{eq:Phi} admits an explicit solution. For example in the Allen-Cahn case $f(x)=x-x^3$ one has
\begin{align}
    \Phi_t(z)=\mathrm{sgn}(z)\frac{e^t}{\sqrt{z^{-2}-1+e^{2t}}}.
\end{align}
\end{remark}
This scheme corresponds to the semi-discrete splitting scheme considered in \cite{BrehierJH18}. 
One can rewrite the scheme as a classical Euler scheme for an auxiliary SPDE.
To that aim, define the auxiliary function 
\begin{align}\label{eq:gh}
    g_{t}(z)=\frac{\Phi_{t}(z)-z}{t},\quad t>0,\quad g_{0}(z):=f(z).
\end{align}
Using the definition of $g_{h}$, $X^{M,N}$ can equivalently be written in the form
\begin{align}\label{eq:approx-discr}
    X^{M,N}_{t_{k+1}}=P_{h}^{N}X^{M,N}_{t_{k}}+hP_{h}^{N}g_{h}(X^{M,N}_{t_{k}})+O^{N}_{t_{k+1}}-P^{N}_{h}O^{N}_{t_{k}}.
\end{align}
Equivalently, we can write a ``mild'' version of the approximation $X^{M,N}$ and extend it to arbitrary points $t\in[0,T]$ as
\begin{align}\label{eq:num-approx-cont}
    X_{t}^{M,N}=P_{t}^{N}u_{0}+\int_{0}^{t}P_{t-k_{M}(s)}^{N}g_{h}(X^{M,N}_{k_{M}(s)})ds+O_{t}^{N},
\end{align}
which agrees with the inductive form \eqref{eq:splitting} on the time grid points $t=t_{k}$, $k=0,\ldots,M$.

\begin{theorem} \label{thm:main2}
Let \cref{ass:f} hold.
    Let $u$ be the unique mild solution of \eqref{eq:SPDE} and for any $M,N\in\N$, let $X^{M,N}$ be as above.
    Then for any $\epsilon>0$ and $p\in[1,\infty)$ there exists a constant $C=C(T,\epsilon, p, K,\cM)$
    such that 
   \begin{align}
      \big(\E\sup_{t\in[0,T]}\norm{u_{t}-X^{M,N}_{t}}_{L^{2}(\mathbb{T})}^p\big)^{1/p}\leq C \big(N^{-1/2+\epsilon}+M^{-1+\epsilon}\big).
   \end{align}
\end{theorem}
The proof of \cref{thm:main2} is given in \cref{sec:poly-f}.
\begin{remark}
    Let $f$ satisfy  \cref{ass:f} (a) with $m=0$.
    In this case $f$ is globally Lipschitz continuous with at most linear growth.
    It is plausible to expect that \cref{thm:main} extends to this case and the splitting scheme is not necessary.
\end{remark}

Next, we show optimality (up to $\epsilon$) of our rates from \cref{thm:main} and \cref{thm:main2}. To that aim, we prove lower bounds. 
To formulate a lower bound, given a set $I\subset\R$, denote
    \begin{equ}
        \mathcal{G}_{M,I}=\sigma(\hat O_{t_k}(n):\,k=0,\ldots,M, n\in I\cap \Z)
    \end{equ}
    for the Fourier transform $\hat O_t=\mathcal{F}(O_t)$.
The scheme $V^{M,N}$, respectively $X^{M,N}$, is $\mathcal{G}_{M,[-N,N]}$-measurable. We now show that no scheme (no matter how computationally expensive) based on the OU increments can achieve higher rates.
\begin{proposition}\label{prop:lb}
    Let $u$ be the solution of \eqref{eq:SPDE} with $u_0=0$ and $f(u)=u$ and $T=1$. 
    Then for sufficiently large $M,N$, one has the lower bound
    \begin{equs}
   & \inf_{Z:\,\mathcal{G}_{M,[-N,N]}\mathrm{-measurable}}\big\|\|u_1-Z\|_{L^2(\T)}\big\|_{L^2(\Omega)}
    \\&\quad=    \big\|\|u_1-\E(u_1|\mathcal{G}_{M,[-N,N]})\|_{L^2(\T)}\big\|_{L^2(\Omega)}\geq(M^{-1}+N^{-1/2})/30.
    \end{equs}
\end{proposition}
\begin{proof}
Let us write $u_1=\sum_{n\in\Z}\hat u_1(n)e_n$ for the Fourier basis $(e_n)$. For any $n\neq \ell,-\ell$, $(\hat u(n),\hat O(n))$ is independent of $(\hat u(\ell),\hat O(\ell))$. In particular, one has $\E(\hat u_1(n)|\mathcal{G}_{M,[N,N]})=\mathbf{1}_{|n|\leq N}\E(\hat u_1(n)|\mathcal{G}_{M,\{-n,n\}})$ and that for nonnegative $n\neq \ell$,  $(\hat u(n),\E(\hat u(n)|\mathcal G_{M,[-N,N]})$ is independent of $(\hat u(\ell),\E(\hat u(\ell)|\mathcal G_{M,[-N,N]})$. Using furthermore the orthogonality of $(e_n)$, one has
\begin{align*}
    \big\|\|u_1-&\E(u_1|\mathcal{G}_{M,[N,N]})\|_{L^2(\T)}\big\|_{L^2(\Omega)}^2=\sum_{n\in\Z}\big\|\hat u_1(n)-\E(\hat u_1(n)|\mathcal G_{M,[-N,N]})\big\|_{L^2(\Omega)}^2
    \\
    &\geq\|\hat u_1(0)-\E(\hat u_1(0)|\mathcal G_{M,\{0\}})\big\|_{L^2(\Omega)}^2+\sum_{n>N}\big\|\hat u_1(n)\|_{L^2(\Omega)}^2.
\end{align*}
Note that the first term above is the best approximation error of an Ornstein-Uhlenbeck process with unit drift, that is $\hat u_1(0)$, given the samples of the driving Wiener process, that is $(\hat O_{t_k}(0))_{k=0,\dots, M}$. By \cite[Theorem 1]{CC} this term is lower bounded by $M^{-2}/12+o(M^{-2})$. As for the second term, since
\begin{equ}
    \big\|\hat u_1(n)\|_{L^2(\Omega)}^2=\int_0^1 e^{8\pi^2 (-n^2+1) t}\,dt\geq \frac{1}{8e\pi^2n^2},
\end{equ}
the sum is lower bounded by $N^{-1}/(8e\pi^2)+o(N^{-1}).$ By the elementary inequality $\sqrt{(a^2+b^2)/2}\geq (a+b)/2$ the claim follows.
\end{proof}

We end this section by providing numerical evidence for the temporal rate $M^{-1}$, that is given in Figure \ref{fig:figure1}. The figure shows a log-log plot of the squared temporal error for the splitting scheme $X^{N,M}$ of the Allen Cahn equation where $f(x)=x-x^3$. We fixed the spatial discretization parameter to be $N=2^{14} -1$. To compute the temporal error, we compared $X^{N,M}$ for $M=2^{k}$, $k=0,\dots,12$ with $X^{N,M_{\text{finest}}}$ for $M_{\text{finest}}=2^{13}$. The slope of the black line, which shows the averaged squared $L^2([-1,1])$-error among the $150$ samples, is approximately equal to $-2$ for large $M$, as suggested by the temporal rate $M^{-1}$.
\begin{figure}[h]
\begin{center}
\includegraphics[scale=0.42]{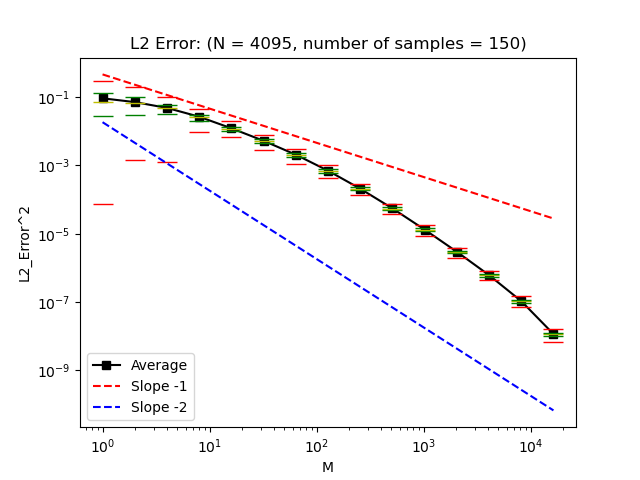}
\caption{}
\label{fig:figure1}
\end{center}
\end{figure}

\section{Preliminaries}\label{sec:prelim}

We introduce the Besov spaces as follows. Let $(\rho_{j})_{j\geq -1}$ be a smooth dyadic partition of unity, i.e. a family of functions $\rho_{j}\in C^{\infty}_{c}(\R)$ for $j\geq -1$, such that 
\begin{itemize}
\item $\rho_{-1}$ and $\rho_{0}$ are non-negative even functions such that the support of $\rho_{-1}$ is contained $B_{1/2}$, the ball of radius $1/2$ around $0$, and the support of $\rho_{0}$ is contained in $B_1\setminus B_{1/4}$;
\item $\rho_{j}(x)=\rho_{0}(2^{-j}x)$, $x\in\R$, $j\geq 0$;
\item $\sum_{j=-1}^{\infty}\rho_{j}(x)=1$ for every $x\in\R$;
\item $\operatorname{supp}(\rho_{i})\cap \operatorname{supp}(\rho_{j})=\emptyset$ for all $\abs{i-j}>1$.
\end{itemize}
The existence of such a partition of unity is classical (see e.g. \cite{Bahouri2011}).
We denote by $\mathcal{S}'(\mathbb{T})$ the space of Schwartz distributions on the torus (i.e. the dual of  $\mathcal{S}(\mathbb{T}):=C^\infty(\T)$).
Note that for any $u\in \mathcal{S}'(\T)$ its Fourier transform is meaningful, and therefore one can define the operators (also known as Littlewood-Paley blocks), $j\geq -1$, 
\begin{align}
    \Delta_j:\mathcal{S}'(\T)\to C^\infty(\T)\qquad \Delta_j u=\mathcal{F}^{-1}\big(k\mapsto \rho_{j}(k)\mathcal{F}(u)(k)\big).
\end{align}
We then define the Besov spaces on the torus for $p,q\in [1,\infty]$, $\theta\in\R$
\begin{align}\label{def:bs}
B^{\theta}_{p,q}:=\big\{ u\in\mathcal{S}'(\mathbb{T}):\norm{u}_{B^{\theta}_{p,q}}=\norm{(2^{j\theta}\norm{\Delta_{j}u}_{L^{p}(\T)})_{j\geq -1}}_{\ell^{q}}<\infty\big\}.
\end{align}
We introduce the shorthand $\calC^{\theta}:=B^\theta_{\infty,\infty}$ for $\theta\in\R$.
We collect the relevant properties of Besov spaces below.
\begin{itemize}
    \item (H\"older spaces) If $\theta\in(0,1)$, then $\calC^{\theta}$ coincides with the space of $\theta$-H\"older continuous functions, (cf.~\cite[Sec. 2.7, Examples]{Bahouri2011}). 
    \item (Embedding) If $1\leq p_{1}\leq p_{2}\leq \infty$, $1\leq q_{1}\leq q_{2}\leq\infty$ and $\theta\in\R$, then $B^{\theta}_{p_{1},q_{1}}$ is continuously embedded in $B_{p_{2},q_{2}}^{\theta-d(1/p_{1}-1/p_{2})}$ (cf. \cite[Proposition 2.71]{Bahouri2011}).
    \item (Derivatives) If $u\in B^{\theta}_{p,q}$ and $n\in\N$, $\norm{\partial^{n}u}_{B^{s-n}_{p,q}}\leq \norm{u}_{B^{s}_{p,q}}$.
    \item (Products) If $\theta,\beta$ are such that  $\theta + \beta > 0$, then for any two distributions $u\in\calC^{\theta}$ and $v\in\calC^{\beta}$ their product $uv$ is well-defined and
there exists a constant  $C=C(\theta,\beta)$, such that the bound
$$ \Vert uv \Vert_{\mathcal{C}^{\min(\theta,\beta)}}\leq C \Vert u\Vert_{\mathcal{C}^{\theta}}\Vert v\Vert_{\mathcal{C}^{\beta}}
$$
holds    (cf.~\cite[Section 2]{Bahouri2011}).
\item (Heat kernel bounds) If $\theta\geq 0$, $\tilde\theta\in [0,2]$, then there exist constants $C_1(\theta)$ and $C_2(\tilde \theta)$ such that for any $\beta\in \R$, $u\in\mathcal{C}^\beta$,  and $t\in(0,1]$, the bounds
\begin{align}\label{eq:HK-1}
     \norm{P_{t}u}_{\calC^{\beta+\theta}}\leq C_1 t^{-\theta/2}\norm{u}_{\calC^{\beta}},\quad \norm{(\operatorname{Id}-P_{t})u}_{\calC^{\beta-\tilde\theta}}\leq C_2 t^{\tilde \theta/2}\norm{u}_{\calC^{\beta}}
\end{align}
hold (cf. \cite[Lemma A.7, A.8]{GIP}).

\end{itemize}

For function spaces on $\R$, we only need the simple notion of $C^k_b$ for $k=0,1,\ldots$, denoting the space of bounded measurable functions whose distributional derivative up to order $k$ are essentially bounded, equipped with the canonical norm (note in particular that elements of $C^0_b$ are not assumed to be continuous).
Moreover denote by $(P_{t}^{\mathbb{R}})_{t\geq 0}$ the heat semigroup on $\mathbb{R}$, that is 
\begin{align}
    P_{t}^{\mathbb{R}}f=p_{t}^{\R}\ast f,\qquad\qquad p_{t}^{\R}(x)=\frac{1}{\sqrt{2\pi t}}e^{-x^2/2t}.
\end{align}
The following estimate is rather immediate: for any $0\leq s\leq t\leq 1$ one has
\begin{align}\label{eq:HK-2}
    \norm{(P_{t}^\R-P_s^\R)u}_{C^0_b}\leq |t-s|\norm{u}_{C^2_b}.
\end{align}

For a Banach space $X$, $C_{T}X$ denotes the space of continuous functions in time with values in $X$ equipped with the supremum norm. For $\gamma\in (0,1]$, define $$C_{T}^{\gamma}X : =\big\{u\in C_TX:\, \norm{u}_{C_{T}^{\gamma}X}:=\sup_{t\in[0,T]}\norm{u_{t}}_{X}+\sup_{0\leq s<t\leq T}\frac{\norm{u_{t}-u_{s}}_{X}}{(t-s)^{\gamma}}<\infty\big\}.$$
Below we use the notation $a\lesssim b$ if there exists a constant $C>0$, such that $a\leq Cb$.
The dependence of the constant $C$ will be clear from the context.
If we want to stress dependence of the constant $C$ on a parameter $\tau$, we write $a\lesssim_{\tau}b$.

We collect a number of simple tools, starting with the proof of Proposition \ref{prop:JK-bad}.
\begin{proof}[Proof of Proposition \ref{prop:JK-bad}]
Note that if $f$ is affine, then \cref{asn:JK} is clearly satisfied. To prove the converse, with the notation $H^s=B^s_{2,2}$ notice that letting $z=(1-\Delta)v\in L^{2}$, \cref{asn:JK} implies that 
    \begin{align}\label{eq:JK-cond}
    \norm{F'(u)z}_{H^{-2}}\leq L\norm{z}_{H^{-2}},\quad \forall z,u\in L^{2}.
    \end{align}
    Note that the Gateaux derivative of $F:u\mapsto f(u)$ is $F':u\mapsto \big(z\mapsto f'(u)z\big)$, so $F'(u)z$ is simply the product $f'(u)z$ \cite[Theorem~2.7]{AP93}.

    We claim that \eqref{eq:JK-cond} implies that $f'$ is constant.
    Assume the contrary and let $a,b$ such that $f'(a)\neq f'(b)$. Set $u=\mathbf{1}_{[0,1/2]}a+\mathbf{1}_{(1/2,1)}b\in L^{2}(\mathbb{T})$, where we identify the torus $\mathbb{T}$ with $[0,1)$ with periodic boundary conditions. Clearly, $h:=f'(u)\notin H^1$.
Notice that for every $k\in\Z$, $z_k:=k^2 e_{k}$ has norm less than $1$ in $H^{-2}$. Therefore \eqref{eq:JK-cond} and $\abs{\langle hz_k,1\rangle}\leq\norm{hz_k}_{H^{-2}}=\sup_{\phi\,:\,\norm{\phi}_{H^{2}}=1}\abs{\langle hz_k,\phi\rangle}$ imply that
\begin{align}
    \|\partial h\|_{L^2}^2=\sum_{k\in\Z}|\langle h,e_k\rangle|^2k^{2}=\sum_{0\neq k\in\Z}|\langle h,k^{2}e_k\rangle|^2k^{-2}\leq L\sum_{0\neq k\in Z}k^{-2}<\infty,
\end{align}
which is a contradiction.
\end{proof}

\begin{proposition}\label{prop:ou}
   Let $(O_{t})_{t\geq 0}$ be as in \eqref{eq:OU}. 
     Let $p\in[1,\infty)$, $\lambda\in (0,1)$ and $\epsilon\in (0,1/2)$. Then there exists a constant $C=C(p,T,\lambda,\epsilon)$, such that for all $0\leq s\leq t\leq T$ one has
    \begin{align}\label{eq:ou-moment-bounds}
        \E\norm{O_{t}-O_{s}}_{\calC^{1/2-\lambda-\epsilon}}^p\leq C (t-s)^{p\lambda/2}.
    \end{align}
    Furthermore, one has
    \begin{align}\label{eq:ou-reg}
        \E\norm{O}_{C_{T}^{\lambda/2}\calC^{1/2-\lambda-\epsilon}}^{p}\leq C.
    \end{align}
Furthermore if $\theta\in[0,1]$, then there exists a constant $C=C(p,T,\theta,\epsilon)$ such that  for all $s,t\in[0,T]$ one has
    \begin{align}\label{eq:s-t-bound}
        \paren{\E\sup_{r\in[0,T]}\norm{(P_{t+s}-P_{t})O_{r}}_{\calC^{-1/2 +\epsilon}}^{p}}^{1/p}\leq C t^{-\theta/2-\epsilon}s^{\frac{1+\theta}{2}}.
    \end{align}
\end{proposition}
\begin{proof}
Note that \eqref{eq:ou-reg} follows from \eqref{eq:ou-moment-bounds}, by virtue of Kolmogorov's continuity theorem (for a version that gives the correct exponent and constant, see e.g. Proposition \ref{prop:vKolmogorov} below with $S_t\equiv \mathrm{Id}$).
Indeed, let $\lambda\in (0,1)$ and $\epsilon\in (0,1/2)$ be arbitrary. Let $\delta\in (0,\epsilon/4)$ be such that $\lambda+2\delta\in (0,1)$ and let $q\geq p$ be large enough so that $\frac{1}{q}<\delta$ and therefore also $\frac{q(\lambda+2\delta)}{2}>1$.
Then take the bound \eqref{eq:ou-moment-bounds} with $\lambda+2\delta\in (0,1)$ in place of $\lambda$, $\epsilon/2\in (0,1/2)$ in place of $\epsilon$, and $q$ in place of $p$.
The obtained bound shows that the condition of Kolmogorov's continuity theorem (Proposition \ref{prop:vKolmogorov}) are satisfied with $X=O$, $V=\calC^{1/2-\lambda-2\delta-\epsilon/2}$, $q$ in place of $p$ therein, $S_t\equiv \mathrm{Id}$, and $\alpha= \frac{q(\lambda+2\delta)}{2}-1$. From the aforementioned properties of the exponents it follows that
 $\frac{\lambda}{2}\in (0,\frac{\alpha}{q})$ and therefore 
the bound \eqref{eq:Kolmogorov-conclusion} holds with the choice $\gamma=\lambda/2$ and $q$ in place of $p$. Using
  that $2\delta+\frac{\epsilon}{2}\leq\epsilon$ and $q\geq p$, this implies \eqref{eq:ou-reg} as written:
  $$\E\norm{O}_{C_{T}^{\lambda/2}\calC^{1/2-\lambda-\epsilon}}^{p}\leq\E\norm{O}_{C_{T}^{\lambda/2}\calC^{1/2-\lambda-2\delta-\epsilon/2}}^{p}\leq\E\norm{O}_{C_{T}^{\lambda/2}\calC^{1/2-\lambda-2\delta-\epsilon/2}}^{q}\leq C'' C.$$ 
 
Furthermore, \eqref{eq:s-t-bound} follows from \eqref{eq:ou-reg} and applying both bounds in \eqref{eq:HK-1}.

It remains to prove \eqref{eq:ou-moment-bounds}. We may and will assume that $p$ is sufficiently large. We then show that for any $j\geq -1$,
\begin{align}\label{eq:j-bound}
 \paren{\E\abs{\Delta_{j}O_{t}(x)-\Delta_{j}O_{s}(x)}^p}^{1/p}\lesssim 2^{j(-1/2+\lambda)}\abs{t-s}^{\lambda/2}.
\end{align}
Indeed, from \eqref{eq:j-bound}, raising to the $p$-th power, multiplying by $2^{pj(1/2-\lambda-\eps/2)}$, and summing over $j$, we get
    \begin{align}
        \E\norm{O_{t}-O_{s}}_{B^{1/2-\lambda-\eps/2}_{p,p}}^p\lesssim (t-s)^{p\lambda/2}.
    \end{align}
Choosing $p$ sufficiently large, \eqref{eq:ou-moment-bounds} follows by the Besov embedding $B_{p,p}^{\theta}\hookrightarrow \calC^{\theta-1/p}$, $\theta\in\R$.

It is left to prove \eqref{eq:j-bound}. Using Gaussian hypercontractivity (i.e. for a Gaussian random variable $Y$ we have that $(\E\abs{Y}^p)^{1/p}\leq C_{p}(\E\abs{Y}^2)^{1/2}$ for any $p\geq 1$), It\^o's isometry, and Parsevel's identity, we have that
\begin{align*}
    \paren{\E\abs{\Delta_{j}O_{t}(x)-\Delta_{j}O_{s}(x)}^p}^{2/p}&\lesssim \E\abs{\Delta_{j}O_{t}(x)-\Delta_{j}O_{s}(x)}^2
    \\&\lesssim \E\abs[\bigg]{\int_{s}^{t}\int_{\T}\Delta_j p_{t-r}(x-y)\xi(dr,dy)}^2\\&\quad+\E\abs[\bigg]{\int_{0}^{s}\int_{\T} \Delta_j (p_{t-r}-p_{s-r})(x-y)\xi(dr,dy)}^2
    \\
    &=\int_s^t\|\Delta_j p_{t-r}(x-\cdot)\|_{L^2(\T)}^2\,dr
    \\
    &\qquad+\int_0^s\|\Delta_j (p_{t-r}-p_{s-r})(x-\cdot)\|_{L^2(\T)}^2\,dr
    \\&=\int_{s}^{t}\sum_{k\in\Z}\rho_{j}(k)^2 e^{-8\pi^{2}k^{2}(t-r)}dr \\&\quad + \int_{0}^{s}\sum_{k\in\Z}\rho_{j}(k)^2 e^{-8\pi^{2}k^{2}(s-r)}(1-e^{-4\pi^{2}k^{2}(t-s)})^2dr.
\end{align*}
Integrating in time and using that $e^{-x}\leq\min(x^{-1},1)$ and $1-e^{-x}\leq \min(x, x^{1/2}, 1)$ for $x\geq 0$, we obtain
\begin{align*}
 \paren{\E\abs{\Delta_{j}O_{t}(x)-\Delta_{j}O_{s}(x)}^p}^{2/p}
    &\lesssim \sum_{k\in\Z} \rho_{j}(k)^2\min(\abs{t-s}, \abs{k}^{-2}) \\&\quad + \sum_{k\in\Z}\rho_{j}(k)^2\min(\abs{k}^{-2}, \abs{k}^{-2}\abs{k}^2\abs{t-s})
    \\&\lesssim 2^{j}\min (\abs{t-s},2^{-2j})\\&\lesssim 2^{j(2\lambda-1)}\abs{t-s}^{\lambda}
\end{align*}
for any $\lambda\in[0,1]$, yielding \eqref{eq:j-bound}.
\end{proof}

For the treatment of the temporal error, we use the stochastic sewing lemma, originating from \cite{Khoa}. We state here a weighted version (see \cite{ABLM, DGL}). The final conclusion of the lemma follows from \cite[Lemma~2.3]{BFG}. Let $$[S,T]_\leq:=\{(s,t)\mid S\leq s<t\leq T\}$$ and $$[S,T]^{\ast}_\leq:=\{(s,t)\mid S\leq s<t< T,\,t-s\leq T-t\}.$$ For a function $\mathcal{A}$ of one variable and $s\leq t$, we write $\mathcal{A}_{st}=\mathcal{A}_{t}-\mathcal{A}_{s}$ and for functions $A$ of two variables and $s\leq u\leq t$, we denote $\delta A_{sut}=A_{st}-A_{su}-A_{ut}$. 
Further, denote by $\E_{s}$ the conditional expectation with respect to $\F_s$.

\begin{lemma}\label{lem:ss}
  Fix $p\geq 2$ and $0\leq S<T\leq 1$. Let $A:[S,T]_{\leq}\to L^{p}(\Omega)$ be such that $A_{st}$ is $\F_{t}$-measurable for all $(s,t)\in[S,T]_{\leq}$. Suppose that there exist $\epsilon_1,\epsilon_2>0,\delta_1,\delta_2\geq 0$ and $C_{1},C_2<\infty$ satisfying $1/2+\epsilon_1-\delta_1>0,\, 1+\epsilon_2-\delta_2>0$ and such that 
  for all $(s,t)\in[S,T]^{\ast}_{\leq}$, $u\in[s,t]$ the bounds
  \begin{align}
      \norm{A_{st}}_{L^{p}(\Omega)}&\leq C_1 (T-t)^{-\delta_1}\abs{t-s}^{1/2+\epsilon_1},\label{eq:SSL-cond1}
      \\
      \norm{\E_{s}[\delta A_{sut}]}_{L^{p}(\Omega)}&\leq C_2 (T-t)^{-\delta_2}\abs{t-s}^{1+\epsilon_2}\label{eq:SSL-cond2}
  \end{align}
  hold. Then there exists a unique $(\F_t)_{t\in[S,T]}$-adapted process $\mathcal{A}:[S,T]\to L^{p}(\Omega)$ such that $\mathcal{A}_{S}=0$ and 
  that there exist $K_1,K_2$ such that for all $(s,t)\in[S,T]^{\ast}_{\leq}$ one has 
  \begin{align}
      \norm{\mathcal{A}_{st}-A_{st}}_{L^{p}(\Omega)}&\leq K_1 (T-t)^{-\delta_1}\abs{t-s}^{1/2+\epsilon_1}+ K_2 (T-t)^{-\delta_2}\abs{t-s}^{1+\epsilon_2},\label{eq:ss1}\\
      \norm{\E_{s}[\mathcal{A}_{st}- A_{st}]}_{L^{p}(\Omega)}&\leq K_2 (T-t)^{-\delta_2}\abs{t-s}^{1+\epsilon_2}\label{eq:ss2}.
  \end{align}
  Furthermore, there exists a constant $C$ depending only on $p,\eps_1,\eps_2$, such that the above bounds hold with $K_{1}=CC_1, K_{2}=C C_2$.
  Finally, there exists a constant $C'$ depending only on $p,\eps_1,\eps_2,\delta_1,\delta_2$, such that for all $(s,t)\in[S,T]_\leq$ one has
  \begin{align}
      \norm{\mathcal{A}_{st}}_{L^{p}(\Omega)}&\leq C'(C_1 \abs{t-s}^{1/2+\epsilon_1-\delta_1}+ C_2 \abs{t-s}^{1+\epsilon_2-\delta_2}).
  \end{align}
\end{lemma}

Next, we give the result on the well-posedness of the reaction-diffusion equation \eqref{eq:SPDE}.

\begin{proposition}\label{prop:wp}
Let \cref{asn:easier} or \cref{ass:f} hold.
Then there exists a unique mild solution $u$ 
to \eqref{eq:SPDE}. Moreover, for any $\lambda\in (0,1)$, $\epsilon\in (0,1/2)$, $p\geq 1$ there exists a constant $C=C(T,p,\lambda,\epsilon,m,K)$ such that the solution satisfies the bound
\begin{align}\label{eq:reg-b}
    \E\norm{u}_{C_{T}^{\lambda/2}\calC^{1/2-\lambda-\epsilon}(\T)}^{p}\leq C(1+\E\norm{u_{0}}_{\calC^{1/2}}^{(2m+1)p}).
\end{align}
\end{proposition}
\begin{proof}
We give the proof under the assumption that $f\in C^{2}$ satisfies \eqref{eq:f-ass1} and \eqref{eq:f-ass2} for $i=0,1,2$, which is in particular implied by \cref{asn:easier} and \cref{ass:f}. These conditions on $f$ imply \cite[Hypothesis 6.2]{Cerrai}. Furthermore, \cite[Hypothesis 6.1]{Cerrai} is satisfied by the operators $A=\Delta$ and $Q=\Id$ with the spaces $H=L^{2}(\mathbb{T})$, $E=C(\mathbb{T})$.
Thus, the existence and uniqueness of a mild solution $u$ with paths in $C_{T}C(\mathbb{T})$ follows from \cite[Proposition 6.2.2]{Cerrai}, as well as the bound, for any $q\geq 1$, 
\begin{align}\label{eq:cerrai-bound}
  \E\norm{u}_{C_{T}C(\mathbb{T})}^q \lesssim (1+\E\norm{u_{0}}_{C(\mathbb{T})}^q).
\end{align}
As an immediate consequence,
\begin{align}\label{eq:f(u)-bound}
      \E\norm{f(u)}_{C_{T}C(\mathbb{T})}^q \lesssim (1+\E\norm{u_{0}}_{C(\mathbb{T})}^{(2m+1)q}).
\end{align}
As for the bound \eqref{eq:reg-b}, note that this bound is shown for $O$ in \eqref{eq:ou-reg}, and one also has immediately from \eqref{eq:HK-1}
\begin{align*}
    \|P_tu_0-P_su_0\|_{\calC^{1/2-\lambda-\eps}}\lesssim |t-s|^{(\lambda+\eps)/2}\|u_0\|_{\calC^{1/2}},
\end{align*}
so it suffices to prove \eqref{eq:reg-b} for $v=u-O-P_\cdot u_0$. One can decompose the increments of $v$ as
 \begin{align*}
     v_t-v_s = \int_{s}^{t}P_{t-r}f(u_{r})dr -\int_{0}^{s} P_{s-r}(\Id-P_{t-s})f(u_{r})dr.
 \end{align*} 
Using again the semigroup estimates \eqref{eq:HK-1}, we obtain for $\lambda\in (0,1)$, $\epsilon\in (0,1/2)$ 
 \begin{align}\label{eq:time-reg}
     \norm{v_t -v_s}_{\calC^{1/2-\lambda-\epsilon}}&\lesssim \int_{s}^{t}\norm{P_{t-r}f(u_{r})}_{\calC^{1/2-\lambda-\epsilon}}dr + \int_{0}^{s}\norm{P_{s-r}(\operatorname{Id}-P_{t-s})f(u_r)}_{\calC^{1/2-\lambda-\epsilon}}dr\nonumber 
     \\&\lesssim
     \Big(\int_s^t(t-r)^{(\lambda/2+\epsilon/2-1/4)\wedge0}\,dr+\int_0^s(t-s)^{\lambda/2+\eps/2}(s-r)^{-1/4}\,dr\Big)\nonumber
     \\&\qquad\times\norm{f(u)}_{C_{T}C(\mathbb{T})}\nonumber
     \\
     &\lesssim (t-s)^{\lambda/2+\eps/2}\norm{f(u)}_{C_{T}C(\mathbb{T})}.\nonumber
 \end{align}
 Taking $p$-th moment, using \eqref{eq:f(u)-bound}, and applying Kolmogorov's continuity theorem, we get the claimed bound \eqref{eq:reg-b}.
\end{proof}

We give a version of Gr\"onwall's inequality that is repeatedly used in \cref{sec:bounded-f}. For a proof see in the Appendix.
\begin{proposition}\label{prop:Gronwall}
    Let $V$ be a Banach space, $p\geq 1$, and take three processes $X,Y,Z$ belonging to $L^p(\Omega;C([0,T];V)).$ Assume furthermore that there exists a Lipschitz continuous function $F$ on $V$ with Lipschitz constant $L_1$, a family $(S(s,t))_{0\leq s\leq t\leq T}$ of uniformly bounded linear operators on $V$ with uniform bound $L_2$ and such that $(s,t)\mapsto S(s,t)v$ is measurable for any $v\in V$, and a measurable mapping $\tau:[0,T]\to [0,T]$ such that $\tau(s)\leq s$ and that the following equality holds for all $0\leq t\leq T$:
    \begin{align}
        X_t-Y_t=Z_t+\int_0^tS(s,t)\big(F(X_{\tau(s)})-F(Y_{\tau(s)})\big)\,ds.
    \end{align}
    Then there exists a constant $C=C(p,L_1,L_2,T)$ such that
    \begin{align}
        \E\sup_{t\in[0,T]}\|X_t-Y_t\|^p\leq C\E\sup_{t\in[0,T]}\|Z_t\|^p.
    \end{align}
\end{proposition}

\section{Proof of Theorem \ref{thm:main}}\label{sec:bounded-f}
We start by giving a brief overview of the proof of Theorem \ref{thm:main}.
Introduce auxiliary processes $\tilde{U}^{N}$ and $\tilde{V}^{M,N}$. Let $\tilde{U}^{N}$ be defined as the solution of 
\begin{align}
    \tilde{U}^{N}_{t}&=P^{N}_{t}u_{0}^{N}+\int_{0}^{t}P_{t-s}^{N}f(\tilde{U}^{N}_{s})ds+O_{t},\quad t\in[0,T]\label{eq:tilde-UN}
\end{align} and let $\tilde{V}^{M,N}$ be defined on the time grid points via the inductive form: 
$\tilde{V}^{M,N}_{0}=\Pi_N u_0$ and
\begin{align*}
\tilde{V}^{M,N}_{t_{k+1}}=P^{N}_{h}\tilde{V}^{M,N}_{t_k}+\Delta_{N}^{-1}(P^{N}_{h}-\operatorname{Id})\Pi_{N}(f(\tilde{V}^{M,N}_{t_{k}}))+O_{t_{k+1}}-P^{N}_{h}O_{t_{k}}
\end{align*} for $k=0,\dots, M-1$, and for $t\in[0,T]$ via
\begin{align}
\tilde{V}^{M,N}_{t}&=P^{N}_{t}u_{0}^{N}+\int_{0}^{t}P_{t-s}^{N}f(\tilde{V}^{M,N}_{k_{M}(s)})ds+O_{t}.\label{eq:tilde-VMN}
\end{align}
First, we address the wellposedness for the equations for $U^{N}$ and $\tilde{U}^{N}$ in the following proposition.
\begin{proposition}
Let $f$ satisfy \cref{asn:easier}. Then there exist unique solutions $U^{N}$ and $\tilde{U}^{N}$ to the equations \eqref{eq:UN} and \eqref{eq:tilde-UN}. Moreover, for any $\lambda\in (0,1)$, $\epsilon\in (0,1/2)$, $p\geq 1$ there exists a constant $C=C(T,p,\lambda,\epsilon)$ such that the solutions satisfy 
\begin{align*}
    \sup_{N\in\N} \E\norm{U^{N}}_{C_{T}^{\lambda/2}\calC^{1/2-\lambda-\epsilon}(\T)}^{p}+\sup_{N\in\N} \E\norm{\tilde{U}^{N}}_{C_{T}^{\lambda/2}\calC^{1/2-\lambda-\epsilon}(\T)}^{p}\leq C(1+\E\norm{u_{0}}_{\calC^{1/2}}^{p}).
\end{align*}
\end{proposition}
\begin{proof}
The argument is classical. Using the global Lipschitz bound on $f$ from \cref{asn:easier}, that $\sup_{N}\norm{P_{t}^{N} u}_{L^{\infty}}\leq \norm{u}_{L^{\infty}}$ for all $t\geq 0$ and applying the Banach fixed point theorem, one finds a unique fixed point of the mild formulation of the equations for $U^{N}$ and for $\tilde{U}^{N}$ in $C_{T}L^{\infty}$ if $T$ is chosen small enough. Patching the solutions on subintervals together yields a solution for any arbitrary time horizon $T>0$. Then plugging the solution back in the mild formulation and using the semigroup estimates for $P^{N}$ instead of $P$, one obtains the claimed regularity bounds (cf.~in the proof of \cref{prop:wp}).
\end{proof}
The error $u-V^{M,N}$ is decomposed into the spatial errors $u-U^N$, $U^N-\tilde U^N$, $V^{M,N}-\tilde{V}^{M,N}$ and the temporal error $\tilde U^N-\tilde{V}^{M,N}$. Bounding the Galerkin error $u-U^{N}$ by order $N^{-1/2+\eps}$ is fairly standard and is already done in e.g. \cite{Jentzen-Kloden}, although a noteworthy difference is that below we obtain estimates in $L^\infty(\T)$ instead of $L^2(\T)$ (Lemmas \ref{lem:aux} and \ref{lem:spatial}).
The biggest novelty of the section comes from the treatment of the temporal error in $\tilde U^N-\tilde{V}^{M,N}$, Lemma \ref{lem:main} below. These three lemmas together imply Theorem \ref{thm:main}.

\begin{lemma}\label{lem:aux}
Assume the setting of Theorem \ref{thm:main}.
Let $N\in\N$ and let $U^{N},\tilde{U}^{N}$ be as in \eqref{eq:UN} and \eqref{eq:tilde-UN} and $V^{M,N},\tilde{V}^{M,N}$ be as in \eqref{eq:JK-scheme} and \eqref{eq:tilde-VMN}. Let $p\in[1,\infty)$ and $\eps>0$. Then there exists a constant $C=C(T,p,\eps,K)$ such that the following bound holds 
\begin{align*}
\big(\E\sup_{t\in[0,T]}\norm{\tilde{U}^{N}_{t}-U^{N}_{t}}_{L^{\infty}}^p\big)^{1/p}+\big(\E\sup_{t\in[0,T]}\norm{\tilde{V}^{M,N}_{t}-V^{M,N}_{t}}_{L^{\infty}}^p\big)^{1/p}\leq C N^{-1/2+\epsilon}.
\end{align*}
\end{lemma}
\begin{proof}
    We first prove the bound for $\tilde{U}^{N}-U^{N}$.
    We have that
    \begin{align*}
        \tilde{U}^{N}_{t}-U^{N}_{t}=\int_{0}^{t}P^{N}_{t-s}(f(\tilde{U}^{N}_s)-f(U^{N}_{s}))ds + O_{t}-O_{t}^{N}.
    \end{align*}
We will apply \cref{prop:Gronwall} with $X=\tilde U^N$, $Y=U^N$, $Z=O-O^N$, $S(s,t)=P^N_{t-s}$, $\tau(s)=s$. Therefore, it suffices to bound $\E\sup_{t\in[0,T]}\|O_t-O^N_t\|_{L^\infty}^p$.
    The difference of the truncated and full expansion of the Ornstein-Uhlenbeck process we write as follows
    \begin{align*}
        (O_{t}-O_{t}^{N})(x)=\int_{0}^{t}\int_{\mathbb{T}}(p_{t-s}-p^{N}_{t-s})(x-y)\xi(ds,dy),
    \end{align*}
    where 
    \begin{align*}
      (p_{t-s}-p^{N}_{t-s})(x)=  \sum_{\abs{k}> N}e^{-4\pi^2 k^2 (t-s)}e^{2\pi i k x}.
    \end{align*}
    Then we obtain for $O_{r,t}=O_{t}-O_{r}$, $r\leq t$, and $O^{N}_{r,t}$ analoguously defined, that
    \begin{align*}
        (O_{r,t}-O_{r,t}^{N})(x)&=\int_{r}^{t}\int_{\mathbb{T}}(p_{t-s}-p^{N}_{t-s})(x-y)\xi(ds,dy) \\&\qquad - \int_{0}^{r}\sum_{\abs{k}>N}e^{-4\pi^2 k^2 (r-s)}(1-e^{-4\pi^2 k^2 (t-r)})e^{2\pi i k (x-y)} \xi(ds,dy).
    \end{align*}
    With the same steps as in the proof of Proposition \ref{prop:ou}, that is Gaussian hypercontractivity (i.e. for a Gaussian random variable $Y$ we have that $(\E\abs{Y}^p)^{1/p}\leq C_{p}(\E\abs{Y}^2)^{1/2}$ for any $p\geq 1$), It\^o's isometry and Parseval identity, we get 
    \begin{align*}
        \paren{\E\abs{\Delta_{j}(O_{r,t}-O_{r,t}^{N})(x)}^p}^{1/p}
        &\lesssim \paren[\bigg]{\int_{r}^{t}\sum_{\abs{k}>N}\rho_{j}(k)^{2}e^{-8\pi^2 k^2 (t-s)}ds\\&\qquad +\int_{0}^{r}\sum_{\abs{k}>N}\rho_j(k)^2 e^{-8\pi^2k^2 (r-s)}(1-e^{-4\pi^2k^2 (t-r)})^2ds}^{1/2}
        \\&\leq \paren[\bigg]{   \sum_{\abs{k}>N}\frac{1-e^{-8\pi^2 k^2 (t-r)}}{8\pi^2 k^{2}} 
        \\&\qquad +  \sum_{\abs{k}>N}\frac{(1-e^{-8\pi^2 k^2 r})(1-e^{-4\pi^2k^2 (t-r)})^2}{8\pi^2 k^{2}}  }^{1/2}
        \\&\lesssim  \paren[\bigg]{(t-r)^{\epsilon}\sum_{\abs{k}>N}\frac{1}{k^{2-2\epsilon}}}^{1/2}
        \\&\lesssim N^{-1/2+\epsilon'}(t-r)^{\epsilon'/2},
    \end{align*} 
    using that for any $\epsilon'\in [0,1]$,  $1-e^{-x}\leq x^{\epsilon'}$, $x\geq 0$.
On the other hand, one has the trivial uniform in $N$ bound
    \begin{align*}
        \paren{\E\abs{\Delta_{j}(O_{r,t}-O_{r,t}^{N})(x)}^p}^{1/p}
        &\lesssim \paren[\bigg]{\int_{r}^{t}\sum_{\abs{k}>N}\rho_{j}(k)^{2}e^{-8\pi^2 k^2 (t-s)}ds
        \\&\qquad +\int_{0}^{r}\sum_{\abs{k}>N}\rho_j(k)^2 e^{-8\pi^2k^2 (r-s)}(1-e^{-4\pi^2k^2 (t-r)})^2ds}^{1/2}\\&\lesssim \paren[\bigg]{\int_{r}^{t} 2^{j} e^{-8\pi^2 2^{2j} (t-s)} ds 
        \\&\qquad +\int_{0}^{r}2^{j} e^{-8\pi^2 2^{2j} (r-s)}(1-e^{-4\pi^2 2^{2j} (t-r)})^2ds}^{1/2}
        \\&\lesssim 2^{-j/2} .
    \end{align*} 
    Hence, interpolation between the two bounds yields that for any $\epsilon''\in(0,1]$,
    \begin{align}\label{eq:incrementOerror}
    \big(\E\abs{\Delta_{j}(O_{r,t}-O_{r,t}^{N})(x)}^p\big)^{1/p}\lesssim 2^{-j\epsilon''/2}N^{(-1/2+\epsilon')(1-\epsilon'')}(t-r)^{(1-\epsilon'')\epsilon'/2}.        
    \end{align}
    Choosing $\eps',\eps''>0$ sufficiently small and $p$ sufficiently large, we repeat the argument from \eqref{eq:j-bound} to \eqref{eq:ou-reg}, to deduce from \eqref{eq:incrementOerror} the bound
    \begin{align}\label{eq:ou-bound}
        \E[\norm{O-O^{N}}_{C_{T}^{\eps'''}\mathcal{C}^{\eps'''}}^p]^{1/p}\lesssim N^{-1/2+\epsilon}
    \end{align}
    with some small $\eps'''>0$.
   Thus the claim follows from \eqref{eq:ou-bound} and \cref{prop:Gronwall}.

    The bound for $\tilde{V}^{M,N}-V^{M,N}$ is done analogously: since
    \begin{align}
        \tilde V^{M,N}_t-V^{M,N}_t=
        \int_{0}^{t}P_{t-s}^{N}(f(\tilde{V}^{M,N}_{k_{M}(s)}-f(V^{M,N}_{k_{M}(s)}))ds+O_{t}-O^{N}_t,
    \end{align}
    we can use \cref{prop:Gronwall} almost exactly as before, with the only difference being that $\tau(s)=k_{M}(s)$.
\end{proof}
In the $L^{2}$-norm the following spatial error bound is known, see e.g. \cite[equation (5.2) for $\gamma=1/4-\epsilon$, $\lambda_{N}=\pi^2 N^2$]{Jentzen-Kloden}. We prove it for the $L^{\infty}$-norm in the following lemma.
\begin{lemma}\label{lem:spatial}
Assume the setting of \cref{thm:main}.
Let $N\in\N$ and let $u,U^{N}$ be as in \eqref{eq:mild} and \eqref{eq:UN}. Let $p\geq 1$ and $\eps\in(0,1/2)$. Then there exists a constant $C=C(p,\eps,K,\cM)$ such that the following bound holds  

\begin{align*}
\big(\E\sup_{t\in[0,T]}\norm{u_{t}-U^{N}_{t}}_{L^{\infty}}^p\big)^{1/p}\leq C N^{-1/2+\epsilon}.
\end{align*}
\end{lemma}
\begin{proof}
    We have that
    \begin{align*}
        u_{t}-U^{N}_{t}&=(P_t-P^N_t)u_0+\int_{0}^{t}P_{t-s}^{N}[f(u_{s})-f(U^{N}_{s})] ds \\&\qquad + \int_{0}^{t}(P_{t-s}-P^{N}_{t-s})f(u_{s}) ds +  O_{t}-O^{N}_{t}.
    \end{align*}
    First, note that
    \begin{align}\label{eq:ic-contribution}
        \|(P_t-P^N_t)u_0\|_{L^\infty}\leq \|u_0-\Pi_Nu_0\|_{L^\infty}\lesssim N^{-1/2+\eps}\|u_0\|_{\mathcal{C}^{1/2}}.
    \end{align}
    Next, using that $e^{-x}\leq x^{-2+\epsilon}$, $x>0$,
    \begin{align}\label{eq:z-bound}
    \int_{0}^{t}\norm{(P_{t-s}-P^{N}_{t-s})f(u_{s})}_{L^{\infty}}ds&=\int_{0}^{t}\norm{(p_{t-s}-p^{N}_{t-s})\ast f(u_{s})}_{L^{\infty}}ds\nonumber\\&\leq \int_{0}^{t}\norm{p_{t-s}-p^{N}_{t-s}}_{L^{2}}\norm{f(u_{s})}_{L^{2}}ds\nonumber\\&\lesssim \norm{f}_{L^{\infty}(\R)}\int_{0}^{t}\sqrt{\sum_{\abs{k}>N}e^{-8\pi^2k^2 (t-s)}}ds \nonumber\\&\lesssim  \norm{f}_{L^{\infty}(\R)}N^{-3/2+\epsilon}\int_{0}^{t}(t-s)^{-1+\epsilon/2}ds\\&\lesssim  \norm{f}_{L^{\infty}(\R)}N^{-3/2+\epsilon}.
    \end{align}
    Then an application of 
    \cref{prop:Gronwall} for $Z_{t}=(P_t-P^N_t)u_0+\int_{0}^{t}(P_{t-s}-P^{N}_{t-s})f(u_{s}) ds +  O_{t}-O^{N}_{t}$, $\tau(s)=s$, $S(s,t)=P^{N}_{t-s}$ together with \eqref{eq:ic-contribution}, \eqref{eq:z-bound}, and \eqref{eq:ou-bound}, yields the claim. 
\end{proof}

As mentioned, the main effort of the section is devoted in proving rate (almost) $1$ for the temporal error $\tilde U^N-\tilde V^{M,N}$, formulated as follows.
\begin{lemma}\label{lem:main}
Assume the setting of Theorem \ref{thm:main}.
    Then for any $\epsilon>0$ and $p\geq 1$ there exists a constant $C=C(T,\epsilon, p,K,\cM)$ such that 
\begin{align*}
    \big(\E\sup_{t\in[0,T]}\norm{\tilde{U}^{N}_{t}-\tilde{V}^{M,N}_{t}}_{L^{\infty}(\mathbb{T})}^{p}\big)^{1/p}\leq C M^{-1+\epsilon}.
\end{align*}
\end{lemma}
Before the proof of \cref{lem:main} we prove a number of temporal error estimates.
Our strategy is as follows. 
We decompose the error as 
\begin{align}\label{eq:error-decomp}
    \tilde{U}^{N}_{t}-\tilde{V}^{M,N}_{t}&=\int_{0}^{t}P^{N}_{t-s}[f(\tilde{U}^{N}_{s})-f(\tilde{V}^{M,N}_{k_{M}(s)})]ds \nonumber
    \\&=\int_{0}^{t}P^{N}_{t-s}[f(\tilde{U}^{N}_{s})-f(\tilde{U}^{N}_{k_{M}(s)})]ds+\int_{0}^{t}P^{N}_{t-s}[f(\tilde{U}^{N}_{k_{M}(s)})-f(\tilde{V}^{M,N}_{k_{M}(s)})]ds 
\end{align}

The second term on the right-hand side of \eqref{eq:error-decomp} is a buckling term (i.e. treated by Gr\"onwall's lemma).
The first term is the crucial one in determining the temporal rate. First we bound this term when replacing $\tilde{U}^N$ by a simpler process.

\begin{proposition}\label{prop:main-prop}
    Let \cref{asn:easier} hold.
    Let $p\geq 1$ and $\epsilon\in(0,1/2)$.
    Then 
  there exists a constant $C=C(T,p,\epsilon, K, \cM)$ such that for all $0\leq s\leq t\leq R\leq T$ it holds that
    \begin{align}\label{eq:f-goal}
    \MoveEqLeft
        \paren[\bigg]{\E\norm[\bigg]{\int_{s}^{t}P^{N}_{R-r}[f(O_{r}+P^{N}_{r}u_{0})-f(O_{k_{M}(r)}+P^{N}_{k_{M}(r)}u_{0})]dr}_{L^{\infty}}^p}^{1/p}\nonumber\\&\leq C M^{-1+2\epsilon}|t-s|^{1/4+\eps/2}.
    \end{align}
\end{proposition}
\begin{proof}
  To simplify notation denote the shifted OU process by $\tilde{O}_{t}:=O_{t}+P_{t}^{N}u_{0}$, $t\in[0,T]$. 
To prove the desired estimate, it suffices to prove that for any $j\geq -1$, $x\in\T$, $0\leq s\leq t\leq R\leq T$, one has
\begin{align}\label{eq:f-actual-goal}
\MoveEqLeft
    \paren[\bigg]{\E\abs[\bigg]{\int_{s}^{t} \Delta_{j}P_{R-r}^{N}[f(\tilde O_{r})-f(\tilde O_{k_{M}(r)})](x)dr}^p}^{1/p}
    \lesssim 
    M^{-1+2\epsilon} 2^{-j\epsilon}(t-s)^{1/4+\epsilon/2}.
\end{align} 
Indeed, \eqref{eq:f-actual-goal} implies 
\begin{align*}
    \MoveEqLeft
        \paren[\bigg]{\E\norm[\bigg]{\int_{s}^{t}P^{N}_{R-r}[f(\tilde O_{r})-f(\tilde O_{k_{M}(r)})]dr}_{B^{\epsilon}_{p,p}}^p}^{1/p}
        \lesssim 
        M^{-1+2\epsilon}(t-s)^{1/4+\epsilon/2}, 
    \end{align*}
    yielding the claim by using the embedding $B^{\epsilon}_{p,p}\hookrightarrow\calC^{\epsilon-1/p}\hookrightarrow L^{\infty}$ for $p$ large enough.
    
    To prove \eqref{eq:f-actual-goal}, we consider $j\geq-1$, $x\in\T$, $R\leq T$ fixed and apply Lemma \ref{lem:ss}
    to the germ, $0\leq s\leq t\leq R$,
    \begin{align}
        A_{st}=\E_{s}\int_{s}^{t}\Delta_{j}P_{R-r}^{N}[f(\tilde O_{r})-f(\tilde O_{k_{M}(r)})](x)dr.
    \end{align}
    We have that for $s<u<t$,
    \begin{align}
        \delta A_{sut}&=\E_{s}\int_{u}^{t}\Delta_j P_{R-r}^{N}[f(\tilde O_{r})-f(\tilde O_{k_{M}(r)})](x)dr\nonumber\\&\qquad-\E_{u}\int_{u}^{t}\Delta_j P_{R-r}^{N}[f(\tilde O_{r})-f(\tilde O_{k_{M}(r)})](x)dr.
    \end{align}
    Thus we obtain that $\E_{s}[\delta A_{sut}]=0$.
    Therefore the condition \eqref{eq:SSL-cond2} is satisfied with $C_2=0$. We now verify \eqref{eq:SSL-cond1}.
    We claim that, uniformly in $x,j$ and for all $0\leq u<t<R$ with $|t-u|\leq |R-t|$ one has
    \begin{align}\label{eq:claim-blow-up}
        \|A_{ut}\|_{L^p(\Omega)}\lesssim 2^{-j\epsilon}  M^{-1+2\epsilon} (R-t)^{-1/4-\epsilon/2}(t-u)^{1/2+\epsilon}.
    \end{align} 
    
First we consider the case $|t-u|\leq 3 TM^{-1}$. We are going to employ the semigroup estimates and the regularity bounds for the OU process \eqref{eq:ou-reg}. Notice that $\tilde{O}$ satisfies the bounds \eqref{eq:ou-reg}, \eqref{eq:s-t-bound} with $O$ replaced by $\tilde{O}$, since by the semigroup estimates \eqref{eq:HK-1} one has
for all $s,t\in[0,T]$,
\begin{align}\label{eq:u0-bound1}
\norm{P_{t+s}u_{0}-P_{t}u_{0}}_{\calC^{-1/2+\epsilon}}\lesssim_\theta s^{(1+\theta)/2}t^{-\theta/2-\epsilon/2}\norm{u_{0}}_{\calC^{1/2}}
\end{align}
for any 
$\theta\in[0,1]$ and
\begin{align}\label{eq:u0-bound2}
\norm{P_{t+s}u_{0}-P_{t}u_{0}}_{\calC^{-1/2+\theta}}\lesssim_\theta s^{(1-\theta)/2}\norm{u_{0}}_{\calC^{1/2}}
\end{align}
for $\theta\in[0,1]$. Furthermore, for $\theta\in (0,1/2)$, $q\geq 1$, using the equivalence of the norms $\norm{u}_{\calC^{\theta}}\eqsim \norm{u}_{\infty} + \sup_{x\neq y}\frac{\abs{u(x)-u(y)}}{\abs{x-y}^{\theta}}$, as well as \eqref{eq:ou-reg},
we can bound the composition by 
\begin{align}\label{eq:comp-b}
  \E\sup_{r\in[0,R]}\norm{f'(\lambda\tilde{O}_r+(1-\lambda)\tilde{O}_{k_{M}(r)})}_{\calC^{\theta}}^q
  \lesssim \norm{f'}_{C_{b}^{1}}^q\E(1+\norm{\tilde{O}}_{C_T\calC^{\theta}})^q \lesssim_{q,\theta} 1.
\end{align} 
Using these bounds, we obtain for $u,t$ such that $\abs{t-u}\leq 3M^{-1}$,
    \begin{align}
    \MoveEqLeft
    \paren[\bigg]{\E\abs[\bigg]{\int_{u}^{t}\Delta_{j} P_{R-r}^{N}[f(\tilde O_{r})-f(\tilde O_{k_{M}(r)})](x)dr}^{p}}^{1/p}\nonumber
    \\&\lesssim 2^{-j\epsilon}\paren[\bigg]{\E\norm[\bigg]{\int_{u}^{t}P_{R-r}^{N}[f(\tilde O_{r})-f(\tilde O_{k_{M}(r)})]dr}_{\calC^{\epsilon}}^{p}}^{1/p}\nonumber
    \\&\lesssim 2^{-j\epsilon}\int_{u}^{t}(R-r)^{-1/4-\epsilon/2} \nonumber
    \\&\hspace{0.2cm}\times\paren[\bigg]{\E\norm[\bigg]{\paren[\bigg]{\int_{0}^{1}f'(\lambda \tilde O_{r}+(1-\lambda)\tilde O_{k_{M}(r)})d\lambda}(\tilde O_{r}-\tilde O_{k_{M}(r)})}_{\calC^{-1/2}}^p}^{1/p}dr\nonumber
    \\&\lesssim 2^{-j\epsilon}(R-t)^{-1/4-\epsilon/2} (t-u)\nonumber
    \\&\hspace{0.2cm}\times\paren[\bigg]{\E\sup_{r\in[0,T]}\norm[\bigg]{\int_{0}^{1}f'(\lambda \tilde O_{r}+(1-\lambda)\tilde O_{k_{M}(r)})d\lambda}_{\calC^{1/2-\epsilon/2}}^{2p}}^{1/2p}\nonumber
    \\&\hspace{0.2cm}\times\paren[\big]{\E\sup_{r\in[0,T]}\norm{(\tilde O_{r}-\tilde O_{k_{M}(r)})}_{\calC^{-1/2+\epsilon}}^{2p}}^{1/2p}\nonumber
    \\&\lesssim 2^{-j\epsilon}(R-t)^{-1/4-\epsilon/2}(t-u)(r-k_{M}(r))^{1/2-\epsilon}\nonumber
    \nonumber
   \\& \lesssim 2^{-j\epsilon}(R-t)^{-1/4-\epsilon/2}(t-u) M^{-1/2+\epsilon} \nonumber
   \\&\lesssim 2^{-j\epsilon}(R-t)^{-1/4-\epsilon/2}(t-s)^{1/2+\epsilon} M^{-1+2\epsilon},\label{eq:long-smallt-s}
    \end{align}
    using $\abs{t-u}\leq 3TM^{-1}$ in the last inequality. This shows that \eqref{eq:claim-blow-up} holds for $\abs{t-u}\leq 3M^{-1}$.
    
    Next, we consider $\abs{t-u}>3TM^{-1}$. Let $t'=k_{M}(u)+\frac{3T}{M}$ be the second smallest grid point bigger than $u$. Note that this implies that for any $r\geq t'$ one has $r-u\geq 2TM^{-1}$ and $k_M(r)-u\geq(r-u)/2$.
    Then we can decompose
    \begin{align*}
    \MoveEqLeft
        \int_{u}^{t}\Delta_j P_{R-r}^{N}[f(\tilde O_{r})-f(\tilde O_{k_{M}(r)})](x)dr
        \\&=\int_{u}^{t'}\Delta_j P_{R-r}^{N}[f(\tilde O_{r})-f(\tilde O_{k_{M}(r)})](x)dr+\int_{t'}^{t}\Delta_j P_{R-r}^{N}[f(\tilde O_{r})-f(\tilde O_{k_{M}(r)})])(x)dr,
    \end{align*}
    where the first summand we can cope with as above, because $(t'-u)\leq 3TM^{-1}$. The conditional expectation of the second summand, we rewrite as follows
    \begin{align}
    \MoveEqLeft
        \E_{u}\bigg[\int_{t'}^{t}\Delta_j P_{R-r}^{N}[f(\tilde O_{r})-f(\tilde O_{k_{M}(r)})](x)dr\bigg]
        \nonumber\\&= \int_{t'}^{t}\Delta_j P_{R-r}^{N}[(P^{\R}_{Q(r-u)}f)(P_{r-u}\tilde O_{u})-(P^{\R}_{Q(k_{M}(r)-u)}f)(P_{k_{M}(r)-u}\tilde O_{u})](x)dr\label{eq:conditional-easy}
    \end{align} 
    using that
    \begin{align*}
        \tilde O_{r}=P_{r-u}\tilde O_{u}+\int_{u}^{r}\int_{\mathbb{T}}p_{r-v}(\cdot-y)\xi(dv,dy)
    \end{align*}
    and that for random variables $X,Y$ with $X$ being $\F_{u}$-measurable and $Y$ being independent of $\F_{u}$ and centered Gaussian with variance $\sigma$, we have that $\E_{u}[f(X+Y)]=(P^{\R}_{\sigma}f)(X)$, where $P^{\R}$ denotes the heat-semigroup on $\R$. Above we denote by $Q(r-u)=\E[(\int_{u}^{r}\int_{\mathbb{T}}p_{r-v}(x-y)\xi(dv,dy))^2]$ the variance, which only depends on the time distance $r-u$ due to stationarity and also does not depend on $x$. Then we decompose further
    \begin{align}
    \MoveEqLeft
        \int_{t'}^{t}\Delta_j P_{R-r}^{N}[(P^{\R}_{Q(r-u)}f)(P_{r-u}\tilde O_{u})-(P^{\R}_{Q(k_{M}(r)-u)}f)(P_{k_{M}(r)-u}\tilde O_{u})](x)dr\nonumber
        \\&=\int_{t'}^{t}\Delta_j P_{R-r}^{N}[(P^{\R}_{Q(r-u)}f)(P_{r-u}\tilde O_{u})-(P^{\R}_{Q(r-u)}f)(P_{k_{M}(r)-u}\tilde O_{u})](x)dr\label{eq:s1}
        \\&\quad+\int_{t'}^{t}\Delta_j P_{R-r}^{N}[[(P^{\R}_{Q(r-u)}f)-(P^{\R}_{Q(k_{M}(r)-u)}f)](P_{k_{M}(r)-u}\tilde O_{u})](x)dr.\label{eq:s2}
    \end{align}
    For the first summand \eqref{eq:s1} we
    use that, similarly to \eqref{eq:comp-b} one has
    \begin{align}\label{eq:comp-b2}
  \E\sup_{r\in[0,R]}\norm{(P^{\R}_{Q(r-u)}f)'(\lambda P_{r-u}\tilde O_{u}+(1-\lambda)P_{k_{M}(r)-u}\tilde O_{u})}_{\calC^{\theta}}^q
  \lesssim_{q,\theta} 1.
\end{align} 
Using this and the bound \eqref{eq:s-t-bound} for $\tilde{O}$ we obtain
    \begin{align}
        &\paren[\bigg]{\E\abs[\bigg]{\int_{t'}^{t}\Delta_j P_{R-r}^{N}[(P^{\R}_{Q(r-u)}f)(P_{r-u}\tilde O_{u})-(P^{\R}_{Q(r-u)}f)(P_{k_{M}(r)-u}\tilde O_{u})](x)dr}^{p}}^{1/p}\nonumber
        \\&\lesssim 2^{-j\epsilon}\paren[\bigg]{\E\norm[\bigg]{\int_{t'}^{t} P_{R-r}^{N}\bigg[\paren[\bigg]{\int_{0}^{t}(P^{\R}_{Q(r-u)}f)'(\lambda P_{r-u}\tilde O_{u}+(1-\lambda)P_{k_{M}(r)-u}\tilde O_{u})d\lambda} \nonumber
        \\&\hspace{5cm}\times(P_{r-u}\tilde O_{u}-P_{k_{M}(r)-u}\tilde O_{u})\bigg]dr}_{\calC^{\epsilon}}^{p}}^{1/p}\nonumber
        \\&\lesssim 2^{-j\epsilon}\int_{t'}^{t}\paren[\bigg]{(R-r)^{-1/4-\epsilon/2}\nonumber
        \\&\hspace{1cm}\times\paren[\bigg]{\E\norm[\bigg]{\int_{0}^{t}(P^{\R}_{Q(r-u)}f)'(\lambda P_{r-u}\tilde O_{u}+(1-\lambda)P_{k_{M}(r)-u}\tilde O_{u})d\lambda}_{\calC^{1/2-\epsilon/2}}^{2p}}^{1/2p}}\nonumber
        \\&\hspace{1cm}\times\paren[\big]{\paren{\E\norm{P_{r-u}\tilde O_{u}-P_{k_{M}(r)-u}\tilde O_{u}}_{\calC^{-1/2+\epsilon}}^{2p}}^{1/2p}}dr\nonumber
        \\&\lesssim 2^{-j\epsilon}
        \int_{t'}^{t}(R-r)^{-1/4-\epsilon/2} (r-k_{M}(r))^{1-2\epsilon}(k_{M}(r)-u)^{-1/2+\epsilon}dr\nonumber
        \\&\lesssim 2^{-j\epsilon}M^{-1+2\epsilon}
        \int_{t'}^{t}(R-r)^{-1/4-\epsilon/2} (r-u)^{-1/2+\epsilon}dr\nonumber
        \\&\lesssim 2^{-j\epsilon}M^{-1+2\epsilon} (R-t)^{-1/4-\epsilon/2}(t-u)^{1/2+\epsilon}\label{eq:I1-easy}
    \end{align}
    where we used that $k_{M}(r)-u\geq (r-u)/2$ for $r\in [t',t]$.
    For the second summand \eqref{eq:s2}, we use the following estimate on $Q$, where $u\leq l \leq r$,
    \begin{align}\label{eq:Q-est}
    \MoveEqLeft
        Q(r-u)-Q(l-u)\nonumber\\&=\E\bigg[\paren[\bigg]{\int_{u}^{r}\int p_{r-t}(\cdot-y)\xi(dt,dy)}^2\bigg]-\E\bigg[\paren[\bigg]{\int_{u}^{l}\int p_{l-t}(\cdot-y)\xi(dt,dy)}^2\bigg]\nonumber
        \\&=\int_{u}^{r}\norm{p_{r-t}}_{L^{2}(\mathbb{T})}^{2}dt-\int_{u}^{l}\norm{p_{l-t}}_{L^{2}(\mathbb{T})}^{2}dt
        =\int_{l-u}^{r-u}\|p_s\|_{L^2(\T)}^2\,ds\nonumber
        \\
        &\leq \int_{l-u}^{r-u}\|p_s\|_{L^1(\T)}\|p_s\|_{L^\infty(\T)}\,ds
        \lesssim \int_{l-u}^{r-u} s^{-1/2}\,ds\nonumber
        \\
        &\lesssim (r-l)^{1-\epsilon}(l-u)^{-1/2+\eps}.
    \end{align}
     The latter bound together with the heat kernel estimate \eqref{eq:HK-2}  then yields
     \begin{align*}
         \norm{(P^{\R}_{Q(r-u)}f)-(P^{\R}_{Q(k_{M}(r)-u)}f)}_{L^{\infty}(\R)}&\lesssim (Q(r-u)-Q(k_{M}(r)-u))\\&\lesssim (k_{M}(r)-u)^{-1/2+\epsilon}(r-k_{M}(r))^{1-\epsilon}.
     \end{align*}
    Hence we obtain
    \begin{align}\label{eq:Q-bound}
    \MoveEqLeft
        \paren[\bigg]{\E\abs[\bigg]{\int_{t'}^{t}\Delta_j P_{R-r}^{N}[[(P^{\R}_{Q(r-u)}f)-(P^{\R}_{Q(k_{M}(r)-u)}f)](P_{k_{M}(r)-u}O_{u}+P^{N}_{k_{M}(r)}u_{0})](x)dr}^p}^{1/p}\nonumber
        \\&\lesssim 2^{-j\epsilon}\int_{t'}^{t}(R-r)^{-\epsilon/2}\norm{(P^{\R}_{Q(r-u)}f)-(P^{\R}_{Q(k_{M}(r)-u)}f)}_{L^{\infty}(\R)}dr\nonumber
        \\&\lesssim  2^{-j\epsilon}(R-t)^{-\epsilon/2}\int_{t'}^{t}(k_{M}(r)-u)^{-1/2+\epsilon}(r-k_{M}(r))^{1-\epsilon}dr\nonumber
        \\&\lesssim 2^{-j\epsilon}M^{-1+\epsilon} (R-t)^{-\epsilon/2}(t-s)^{1/2+\epsilon},
    \end{align}
    using again that $k_{M}(r)-u\geq (r-u)/2$ for $r\in [t',t]$. We therefore get \eqref{eq:claim-blow-up}.

    Therefore Lemma \ref{lem:ss} finishes the proof of \eqref{eq:f-actual-goal} provided we justify that
    \begin{align*}
        \mathcal{A}_t=\int_{0}^{t}\Delta_j P^{N}_{R-r}[f(\tilde O_r)-f(\tilde O_{k_{M}(r)})](x)dr.
    \end{align*}
    To this end, we need to verify \eqref{eq:ss1} and \eqref{eq:ss2}, the latter of which is trivial since $\E_s(\mathcal{A}_{st}-A_{st})=0$. The former is trivial for another reason: from the boundedness of $f$ one immediately gets $|\mathcal{A}_{st}-A_{st}|\leq 2(t-s)\|f\|_{L^\infty}$. The proof is finished.   
\end{proof}

In the proof of the following corollary, we use a variant of Kolmogorov's continuity theorem, that we state for completeness and prove in the appendix.
\begin{proposition}\label{prop:vKolmogorov}
Let $(X_t)_{t\in[0,T]}$ be a continuous stochastic process starting from $0$ with values in a Banach space $V$ and let $(S_t)_{t\geq 0}$ be strongly continuous semigroup. Then, if for some $p>0$, $\alpha>0$, $C'<\infty$ it holds for all $0\leq s\leq t\leq T$ that
\begin{align}
    \E\|X_t-S_{t-s}X_s\|^p\leq C'|t-s|^{1+\alpha},
\end{align}
then one has for all $\gamma\in(0,\alpha/p)$
\begin{align}\label{eq:Kolmogorov-conclusion}
    \E\big(\sup_{s< t\in[0,T]}|t-s|^{-\gamma }\|X_t-S_{t-s}X_s\|\big)^p\leq C''C',
\end{align}
where $C''$ depends only on $p,\gamma,\alpha,T$, and the semigroup $(S_t)_{t\geq0}$.
\end{proposition}

\begin{corollary}\label{cor:beforeGirs-temp}
Let \cref{asn:easier} hold.
     Let $p\geq 1$ and $\epsilon\in(0,1/4)$.
     Then 
  there exists a constant $C=C(T,p,\epsilon, K, \cM)$ such that  
    \begin{align}
    \MoveEqLeft
        \paren[\bigg]{\E\sup_{R\in[0,T]}\norm[\bigg]{\int_{0}^{R}P^{N}_{R-s}[f(O_{s}+P^{N}_{s}u_{0})-f(O_{k_{M}(s)}+P^{N}_{k_{M}(s)}u_{0})]ds}_{L^{\infty}}^p}^{1/p}\leq C M^{-1+\epsilon}.
    \end{align}
\end{corollary}
\begin{proof}
We apply \cref{prop:vKolmogorov}. By \cref{prop:main-prop} (with $\eps/2$ in place of $\eps$ therein) the process
\begin{align}
    X_t=\int_{0}^{t}P^{N}_{t-s}[f(O_{s}+P^{N}_{s}u_{0})-f(O_{k_{M}(s)}+P^{N}_{k_{M}(s)}u_{0})]ds,
\end{align}
satisfies the conditions of \cref{prop:vKolmogorov} with the semigroup $S=P^N$, $V=L^\infty$, any $p\geq8$, $\alpha=1/16$, and $C'=(C M^{-1+\epsilon})^p$.
\end{proof}

\begin{corollary}\label{cor:main-temp}
    Let \cref{asn:easier} hold.
    Let $p\geq 1$ and $\epsilon\in(0,1/2)$.
    Let $\tilde U^N$ be the solution of \eqref{eq:tilde-UN}. Then 
  there exists a constant $C=C(T,p,\epsilon, K, \cM)$ such that  
    \begin{align}\label{eq:f-goal-2}
    \MoveEqLeft
        \paren[\bigg]{\E\sup_{R\in[0,T]}\norm[\bigg]{\int_{0}^{R}P^{N}_{R-s}[f(\tilde U^N_s)-f(\tilde U^N_{k_{M}(s)})]ds}_{L^{\infty}}^p}^{1/p}\leq C M^{-1+\epsilon}.
    \end{align}
\end{corollary}

\begin{proof}
    We follow the proof of Lemma 2.3.6 in \cite{BDG-SPDE}.
    Define the probability measure $\mathbb{Q}$ via
    \begin{align*}
        \frac{d\mathbb{Q}}{d\p}=\rho=\exp\paren[\bigg]{-\int_{0}^{T}\int f(\tilde{U}^N_s(y))\xi(dy,ds)-\frac{1}{2}\int_{0}^{T}\int \abs{f(\tilde{U}^N_s(y))}^2dyds}.
    \end{align*}
    Girsanov's theorem (\cite[Theorem 10.14]{DPZ}) gives that $\xi(dy,ds)+f(\tilde{U}^N_s(y))dyds$ defines a space-time white noise measure under $\mathbb{Q}$ independent of $\F_0$. This implies that the law of $(\tilde{U}^{N}_t)_{t\in[0,T]}$ under $\mathbb {Q}$ coincides with the law of  $(O_t+P_{t}^{N}u_{0})_{t\in[0,T]}$ under $\p$. An easy exercise shows that $\E[\rho^{-1}]\lesssim C(\norm{f}_{L^{\infty}(\R)})<\infty$.
    Therefore, defining a function $g$ on  the space of continuous functions $Z$ by 
    \begin{align}
       g(Z):= \sup_{R\in[0,T]}\norm[\bigg]{\int_{0}^{R}P^{N}_{R-s}[f(Z_s)-f(Z_{k_{M}(s)})]ds},
    \end{align}
    we can bound
    \begin{align*}
    \E[\abs{g(\tilde{U}^{N})}^{p}]=\E[\rho\rho^{-1}\abs{g(\tilde{U}^{N})}^{p}]&=\E_{\mathbb{Q}}[\rho^{-1}\abs{g(\tilde{U}^{N})}^{p}]
    \\&\leq \E_{\mathbb{Q}}[\rho^{-2}]^{1/2}\E_{\mathbb{Q}}[\abs{g(\tilde{U}^{N})}^{2p}]^{1/2}
    \\&= \E[\rho^{-1}]^{1/2}\E[\abs{g(O+P_{\cdot}u_{0}^{N})}^{2p}]^{1/2}
    \\&\lesssim \E[\abs{g(O+P_{\cdot}u_{0}^{N})}^{2p}]^{1/2}.
    \end{align*}
    The proof is finished by applying \cref{cor:beforeGirs-temp} with $2p$ in place of $p$ to bound the right-hand side.
\end{proof}

We can now prove \cref{lem:main}, which finalizes the proof of \cref{thm:main}.
\begin{proof}[Proof of \cref{lem:main}]
By \eqref{eq:error-decomp}, we can apply \cref{prop:Gronwall} with $X=\tilde U^N$, $Y=\tilde V^{M,N}$,
\begin{align}
    Z_t=\int_{0}^{t}P^{N}_{t-s}[f(\tilde U^N_s)-f(\tilde U^N_{k_{M}(s)})]ds,
\end{align}
$S(s,t)=P^N_{t-s}$, and $\tau(s)=k_M(s)$. Using \cref{cor:main-temp} to bound $Z$, we get the claim.
\end{proof}

\section{A priori bounds for superlinear $f$}\label{sec:apriori}
 
 In this section, we prepare for the error analysis in case of a superlinearily growing nonlinearity $f$ that satisfies \cref{ass:f}.
 One of the steps that become less obvious (and, for too naive approximations simply impossible \cite{beccari2019strong}), is to bound the approximations uniformly in $N,M$. The purpose of this section is to have such a priori estimates on $X^{N,M}$ as well as a couple of related processes.

Recall $\Phi_h$, $g_h$, and $X^{M,N}$ introduced in \cref{sec:formulation}. We now introduce some further auxiliary processes.
First note that one can also view $X^{M,N}$ as the standard Euler discretization for the SPDE with nonlinearity $g_h$, cf.~\eqref{eq:num-approx-cont}. 
The mild solution of the SPDE with nonlinearity $g_h$ is given by
\begin{align}\label{eq:auxSPDE-mild}
    X^h_t=P_tu_0+\int_0^tP_{t-s}g_h(X^h_s)ds+O_t.
\end{align}
Furthermore we denote by $X^{h,N}$ its Galerkin approximation, that is,
\begin{align}\label{eq:auxGal}
X^{h,N}_t=P^N_tu_0+\int_0^tP^N_{t-s}g_h(X^{h,N}_s)ds+O^N_t.
\end{align}
First, we derive bounds on the nonlinearities $g_h$, $\Phi_h$. Like $f$, also $g_{h}$ enjoys a local Lipschitz condition, a polynomial growth and a one-sided Lipschitz condition, but $\Phi_{h}$ is globally Lipschitz continuous. 
The proof of the following lemma can be found in the Appendix.
\begin{lemma}\label{lem:g_h-bounds}
Let $f$ satisfy \cref{ass:f} (a). Let $\Phi_h$, $g_h$ be given as in \eqref{eq:Phi}, \eqref{eq:gh}
Then there exist constants $C, \tilde K>0$ and $\tilde{m}\geq m$, depending only on $K$ and $m$, such that for $i=0,1,2,3$ and for all $h\in[0,1]$, the functions $\Phi_{h}$ and $g_{h}$ satisfy
\begin{align*}
    \abs{\Phi_{h}(x)-\Phi_{h}(y)}&\leq e^{Kh/2}\abs{x-y},\quad\forall x,y\in\R\\
    \abs{\partial^{i}g_{h}(x)}&\leq \tilde K(1+\abs{x}^{2\tilde{m}+1-i}),\quad\forall x\in\R,\quad i=0,1,2,3\\
    \partial g_{h}(x)&\leq K,\quad\forall x\in\R\\
    \abs{g_{h}(x)-g_{0}(x)}&\leq C h(1+\abs{x}^{4m+2}),\quad\forall x\in\R.
\end{align*}   
In particular, $g_{h}$ satisfies 
\begin{align}
    (g_{h}(x)-g_{h}(y))(x-y)&\leq K (x-y)^2,\quad\forall x,y\in\R\label{eq:g-bound1}\\
    \abs{g_{h}(x)-g_{h}(y)}&\leq \tilde K(1+\abs{x}^{2\tilde m}+\abs{y}^{2\tilde m})\abs{x-y},\quad\forall x,y\in\R.\label{eq:g-bound2}
\end{align}
\end{lemma}

Given the uniform bounds on $g_h$, we obtain uniform bounds on the processes $X^h$, $X^{h,N}$, formulated as follows.

\begin{corollary}\label{cor:cont-sol-bounds}
 Let \cref{ass:f} hold.
 Then there exist unique mild solutions
    $X^h$ and $X^{h,N}$
    to the equations \eqref{eq:auxSPDE-mild} and \eqref{eq:auxGal}, respectively.
Moreover, for any $p\geq 1$, $\epsilon\in (0,1/2)$, $\lambda\in (0,1)$, there exists a constant $C=C(T, m,K,\epsilon,\lambda,p,\cM)$ such that the following bound holds 
    \begin{align*}
    \MoveEqLeft
        \sup_{M\in\N}\E\norm{X^{h}}_{C_{T}^{\lambda/2}\calC^{1/2-\lambda-\epsilon}}^p
          +\sup_{M,N\in\N}\E\norm{X^{h,N}}_{C_{T}^{\lambda/2}\calC^{1/2-\lambda-\epsilon}}^p
        \leq C.
    \end{align*} 
\end{corollary}
\begin{proof}
First note that the bound for $X^h$ follows directly from \cref{prop:wp}, since by \cref{lem:g_h-bounds}, $g_h$ satisfies \cref{ass:f} with constants uniform in $h$. The proof for $X^{h,N}$ is almost identical: one only needs to note that \cite[Hypothesis 6.1]{Cerrai} is satisfied also by the operators $A=\Pi_N\Delta$, $Q=\Pi_N$, with the spaces $H=L^{2}(\mathbb{T})$, $E=C(\mathbb{T})$. The rest of the proof of \cref{prop:wp} follows verbatim.
\end{proof}

The a priori bounds for the splitting scheme can be derived using that  $\Phi_{h}$ is globally Lipschitz and the uniform growth bound on $g_{h}$.
\begin{proposition}\label{prop:apriori}
    Let \cref{ass:f} hold.
    Let $p\geq 1$
    and $X^{M,N}$ 
    be as in \eqref{eq:num-approx-cont}. 
    Then there exists a constant $C=C(p, K,m,T,\cM)$ such that the following a priori bound holds
    \begin{align*}
    \MoveEqLeft
       \sup_{M,N\in\N} \E\sup_{k=0,\dots,M}\norm{X^{M,N}_{t_{k}}}_{L^{\infty}}^p
       \leq C.
    \end{align*}
\end{proposition}
\begin{proof}
We follow the proof of \cite[Proposition 3.7]{BrehierG19}.
Let $R^{M,N}_{t_{l}}=X^{M,N}_{t_{l}}-O^{N}_{t_{l}}-P^{N}_{t_{l}}u_0$, $l=0,\dots,M$. Using that $P_h$ is bounded on $L^{\infty}$ with norm $1$, that $\Phi_{h}$ is globally Lipschitz with Lipschitz constant $e^{Kh/2}$, $\Phi_{h}(z)-z=h g_{h}(z)$, and the growth bound for $g_h$, we can write iteratively
\begin{align*}
    \norm{R^{M,N}_{t_{k+1}}}_{L^{\infty}}&=\norm{P_{h}^{N}(\Phi_{h}(R^{M,N}_{t_{k}}+O^{N}_{t_{k}})-\Phi_{h}(O^{N}_{t_{k}})) + P_{h}^{N}(\Phi_{h}(O^{N}_{t_{k}})-O_{t_{k}}^{N})}_{L^{\infty}}
    \\&\leq e^{Kh/2} \norm{R^{M,N}_{t_{k}}}_{L^{\infty}}+h \tilde{K}^p(1+\norm{O^{N}}_{C_{T}L^{\infty}}^{(2\tilde{m}+1)})
    \\
    &\vdots
    \\&\leq e^{(k+1)Kh/2}\norm{\Pi_{N}u_{0}}_{L^{\infty}} + \tilde{K}\sum_{j=0}^{k} e^{jKh/2}h(1+\norm{O}_{C_{T}L^{\infty}}^{(2\tilde{m}+1)}).
\end{align*}
In the last inequality we also estimated $\norm{O^{N}}_{C_{T}L^{\infty}}^p=\norm{\Pi_{N}O}_{C_{T}L^{\infty}}^p\leq \norm{O}_{C_{T}L^{\infty}}^p$.
Note that $(k+1)h\leq Mh=T$.
Therefore,
\begin{align*}
   \sup_{k=0,\dots,M} \norm{R^{M,N}_{t_{k}}}_{L^{\infty}}&\leq e^{KT/2}(\norm{u_{0}}_{L^{\infty}}+M h \tilde{K} (1+\norm{O}_{C_{T}L^{\infty}}^{(2\tilde{m}+1)}))\\&\leq C(
   1+\norm{u_{0}}_{L^{\infty}}+\norm{O}_{C_{T}L^{\infty}}^{(2\tilde{m}+1)}).
\end{align*} 
Taking $p$-th moment expectations and using the bounds on $O$ and $O^N$, we obtain the claim.
\end{proof}
\begin{corollary}\label{cor:apriori}
Let \cref{ass:f} hold.
Let $X^{M,N}$
be as in \eqref{eq:num-approx-cont}.
Then for any $p\geq 1$, $\lambda\in (0,1)$, $\epsilon\in(0,1/2)$,  there exists a constant $C=C(p,T,\lambda,\epsilon, K, m,\cM)$ such that
    \begin{align}
    \MoveEqLeft\label{eq:aprioribound100}
        \sup_{M,N\in\N}\E\norm{X^{M,N}}_{C_{T}^{\lambda/2}\calC^{1/2-\lambda-\epsilon}}^p
        \leq C.
    \end{align} 
    Let $R^{M,N}_t:=X^{M,N}_t-O^N_t-P^{N}_tu_0$, $t\in[0,T]$. Then  for any $\alpha\in (0,2)$, $\epsilon \in (0,1-\alpha/2)$, there exists a constant $C=C(p,T,\epsilon,\alpha, K, m,\cM)$ such that,
    \begin{align*}
        \sup_{M,N\in\N}\E\norm{R^{M,N}}_{C_{T}^{1-\alpha/2-\epsilon}\calC^{\alpha}}^p 
        \leq C.
    \end{align*}
\end{corollary}
\begin{proof}
    The bound \eqref{eq:aprioribound100} follows from \cref{prop:apriori} in exactly the same  way as in \cref{prop:wp} the bound \eqref{eq:reg-b} is derived from \eqref{eq:cerrai-bound}.
    To see the bound on $R^{M,N}$ one can write
    \begin{align*}
        R^{M,N}_{t}-R^{M,N}_{s}&=\int_{s}^{t}P^{N}_{t-k_{M}(r)}g_{h}(X^{M,N}_{k_{M}(r)})dr 
        \\&\qquad + \int_{0}^{s}P_{s-k_{M}(r)}^N(P_{t-s}^N-\operatorname{Id})g_{h}(X^{M,N}_{k_{M}(r)})dr.
    \end{align*}
    We estimate each term by the semigroup estimates \eqref{eq:HK-1} as follows.
First, using $t-k_{M}(r)\geq t-r$, we have
     \begin{align*}
        \norm[\bigg]{\int_{s}^{t}P^{N}_{t-k_{M}(r)}g_{h}(X^{M,N}_{k_{M}(r)})dr}_{\calC^{\alpha}}&\lesssim \int_{s}^{t}(t-k_{M}(r))^{-\alpha/2}\norm{g_{h}(X^{M,N}_{k_{M}(r)})}_{L^{\infty}}dr\nonumber\\&\lesssim (t-s)^{1-\alpha/2}(1+\sup_{k}\norm{X^{M,N}_{t_{k}}}_{L^{\infty}}^{2\tilde{m}+1}).
    \end{align*}
    Second, using $s-k_{M}(r)\geq s-r$, we obtain
    \begin{align*}
    \MoveEqLeft
      \norm[\bigg]{\int_{0}^{s}P_{s-k_{M}(r)}^N(P_{t-s}^N-\operatorname{Id})g_{h}(X^{M,N}_{k_{M}(r)})dr}_{\calC^{\alpha}}\\&\lesssim \int_{0}^{s}(s-k_{M}(r))^{-1+\epsilon} \norm{(P_{t-s}^N-\operatorname{Id})g_{h}(X^{M,N}_{k_{M}(r)})}_{\calC^{-(2-\alpha-2\epsilon)}}dr
      \nonumber\\&\lesssim (t-s)^{1-\alpha/2-\epsilon}(1+\sup_{k}\norm{X^{M,N}_{t_{k}}}_{L^{\infty}}^{2\tilde{m}+1}).
    \end{align*} 
    Taking $p$-th moment and using \cref{prop:apriori} gives the claim.
\end{proof}

\section{Proof of Theorem \ref{thm:main2}}\label{sec:poly-f}
In this section, we prove that the same error bound of $M^{-1+\epsilon}+N^{-1/2+\epsilon}$ can be reached also in the superlinearly growing case.
There are several steps that become significantly more involved.
We mentioned and addressed the question of a priori bounds in \cref{sec:apriori}.
The lack of global Lipschitz bound also requires changes in the buckling, as the mild form of Gr\"onwall's lemma as in \cref{prop:Gronwall} is no longer applicable.
The buckling steps will be therefore performed by switching to the variational formulation, where the one-sided Lipschitzness is easier to exploit.
Finally, the superlinear growth prevents us from appealing to Girsanov's theorem as in \cref{sec:bounded-f}, which is then replaced by using a more involved stochastic sewing argument.

\begin{remark}\label{rem:weak}
As mentioned, we will frequently switch to the the weak/variational form of certain equations. Let us make this more precise. Recall (cf. e.g. \cite{bell}) that if the functions $F,G\in C([0,T]\times \T)$, $H\in C(\T)$ satisfy for all $(t,x)\in [0,T]\times \T$ the mild formulation 
\begin{equ}
    F(t,x)=\int_{\T}p_t(x-y)H(y)dy+\int_0^t\int_\T p_{t-s}(x-y)G(s,y) ds dy,
\end{equ}
then $F$ is the unique weak solution of
\begin{equ}
    \partial_t F=\Delta F+G,\qquad F_0=H,
\end{equ}
and in particular it satisfies the energy inequality
\begin{equ}
    \partial_t\|F(t,\cdot)\|_{L^2}^2\leq 2\langle F(t,\cdot),G(t,\cdot)\rangle,
\end{equ}
where $\langle\cdot,\cdot\rangle$ is the $L^2(\T)$ inner product. As a consequence, for $p\geq 2$ one also has
\begin{equ}
    \partial_t\|F(t,\cdot)\|_{L^2}^p\leq p\|F(t,\cdot)\|_{L^2}^{p-2}\langle F(t,\cdot),G(t,\cdot)\rangle.
\end{equ}
The same conclusion holds if the heat kernel $P$ is replaced by $P^N$.
\end{remark}

Since we deal with growing nonlinearities, we introduce some weighted spaces.
For polynomial weights $\omega(x)=(1+\abs{x}^{2})^{-\beta/2}$, $\beta\geq 1$, we define for $k\in\N_{0}$ the weighted 
H\"older space
$$ C^{k}_{\omega}=\{f\in\mathcal{S}'\mid \omega f\in C^{k}_{b}\}, \quad \norm{f}_{C^{k}_{\omega}}:=\norm{\omega f}_{C^{k}_{b}}.$$
One can easily show the following semigroup estimate on the weighted space: for any $0\leq s\leq t\leq 1$ one has
\begin{align}\label{eq:HK-3}
    \norm{(P_{t}^{\R}-P_{s}^{\R})u}_{C^{0}_{\omega}}\lesssim \abs{t-s} \norm{u}_{C^{2}_{\omega}}.
\end{align}
Notice that, since $g_h$ and its derivatives of order $i=1,2,3$ satisfy a polynomial growth bound, there exists an appropriate weight $\omega$, such that $g_{h}\in C^{3}_{\omega}$. For the remainder of the article we fix such a  weight. 

The error between the true solution $u$ and the numerical solution $X^{M,N}$ can be decomposed as follows. 
\begin{align*}
    \E[\sup_{t\in[0,T]}\norm{u_{t}-X^{M,N}_{t}}_{L^{2}}^p]
    &\lesssim 
    \E[\sup_{t\in[0,T]}\norm{u_{t}-X^{h}_{t}}_{L^{2}}^p]
    +\E[\sup_{t\in[0,T]}\norm{X^{h}_{t}-X^{h,N}_{t}}_{L^{2}}^p] 
    \\&\qquad
    +\E[\sup_{t\in[0,T]}\norm{ X^{h,N}_{t}- X^{M,N}_{t}}_{L^{2}}^p].
\end{align*}
The first term $u-X^h$ is the error coming from replacing $f$ by $g_h$ in the continuum equation. It is estimated in \cref{lem:a}.
The second term $X^h-X^{h,N}$ is the error of the spectral Galerkin approximation of $X^h.$ This is bounded in \cref{lem:b}. The third term $X^{h,N}- X^{M,N}$ is the most crucial contribution to the error and the most challenging to bound by a quantity of order $M^{-1+\eps}+N^{-1/2+\epsilon}$. This is the content of \cref{lem:c}.
These three estimates together yield the main result, \cref{thm:main2}.
\begin{lemma}\label{lem:b}
    Let \cref{ass:f} hold.
    Then for any $p\geq 1$, $\epsilon\in (0,1/2)$, there exists a constant $C=C(T,\epsilon, p, m,K,\cM)$ such that
    \begin{align*}
        \E\sup_{t\in[0,T]}\norm{X^{h}_{t}-X^{h,N}_{t}}_{L^{2}}^p\leq C N^{-p(1/2-\epsilon)}.
    \end{align*}
\end{lemma}
\begin{proof}
Let $R^{h}_t:=X^{h}_t-O_t-P_tu_0$ and $R^{h,N}_t:= X^{h,N}_t-O^N_t-P^{N}_t u_0$, $t\in[0,T]$. 
Then we clearly have
\begin{align}\label{eq:triangle}
    \norm{X^{h}_{t}-X^{h,N}_{t}}_{L^{2}}^{p}\lesssim \norm{R^{h}_{t}-R^{h,N}_{t}}_{L^{2}}^{p} + \norm{O_t-O^{N}_{t}}_{L^{2}}^p+\norm{u_{0}-\Pi_{N}u_{0}}_{L^{2}}^{p} .
\end{align}
The last term can be bounded using the inequality $\norm{\Pi_{N}u-u}_{L^{2}}\lesssim N^{-\alpha}\norm{u}_{H^{\alpha}}\lesssim N^{-\alpha}\norm{u}_{\calC^{\alpha+\epsilon}}$ for any $\alpha\geq 0$, $\epsilon>0$, due to the embedding $\mathcal{C}^{\beta}\subset H^{\beta'}$ for $\beta>\beta'$, and we take $\alpha=1/2-\epsilon$.
The second term is bounded in \eqref{eq:ou-bound}.
For the first term, we use \cref{rem:weak} with $F=R^{h}- R^{h,N}$, $H=0$, and $G=g_h(X^{h})-\Pi_{N}(g_h(X^{h,N}))$. We get that
for $p\geq 2$
\begin{align*}
    \norm{R^{h}_{t}-R^{h,N}_{t}}_{L^{2}}^{p}&\leq p\int_{0}^{t} \norm{R^{h}_s - R^{h,N}_s}_{L^{2}}^{p-2} \langle g_{h}(X^{h}_{s})-\Pi_{N}(g_{h}(X^{h,N}_{s})),R^{h}_{s}-R^{h,N}_{s}\rangle ds.
\end{align*}
Smuggling in 
$g_{h}(X^{h,N}_{s})$, using 
the definition of $R^{h},R^{h,N}$ and 
the bounds for $g_h$ from \cref{lem:g_h-bounds}, as well as Cauchy-Schwartz inequality followed by Young's inequality for $p_{1}=p/(p-1)$ and $p_{2}=p$, we arrive at
    \begin{align*}
    \MoveEqLeft
        \norm{R^{h}_{t}-R^{h,N}_{t}}_{L^{2}}^{p}
        \\&\leq p\int_{0}^{t} \norm{R^{h}_s-R^{h,N}_s}_{L^{2}}^{p-2} \langle g_{h}( X^{h}_{s})-g_{h}( X^{h,N}_{s}),X^{h}_{s}-X^{h,N}_{s}\rangle ds 
        \\&\qquad-p\int_{0}^{t} \norm{R^{h}_s-R^{h,N}_s}_{L^{2}}^{p-2} \langle g_{h}(X^{h}_{s})-g_{h}(X^{h,N}_{s}),O_{s}^{N}-O_{s}\rangle ds 
        \\&\qquad + p\int_{0}^{t} \norm{R^{h}_s-R^{h,N}_s}_{L^{2}}^{p-2} \langle g_{h}( X^{h,N}_{s})-\Pi_{N}(g_{h}( X^{h,N}_{s})),R^{h}_{s}-R^{h,N}_{s}\rangle ds 
        \\& \lesssim  \int _{0}^{t} \norm{R^{h}_s-R^{h,N}_s}_{L^{2}}^{p-2}\norm{X^{h}_{s}-X^{h,N}_{s}}_{L^2}^2 ds 
        \\&\qquad + (1+\norm{X^{h}}_{C_{T}L^{\infty}}^{2\tilde{m}}+\norm{X^{h,N}}_{C_{T}L^{\infty}}^{2\tilde{m}})\\&\hspace{3cm}\times \int _{0}^{t} \norm{R^{h}_s-R^{h,N}_s}_{L^{2}}^{p-2} \norm{X^{h}_{s}-X^{h,N}_{s}}_{L^2}\norm{O_{t}-O^{N}_{t}}_{L^2} ds 
        \\&\qquad+ \int_{0}^{t}\norm{R^{h}_s-R^{h,N}_s}_{L^{2}}^{p}ds + \int_{0}^{t}\norm{g_{h}(X^{h,N}_{s})-\Pi_{N}(g_{h}(X^{h,N}_{s}))}_{L^{2}}^pds.
        \end{align*}
        Using 
        \eqref{eq:triangle} we further obtain
        \begin{align*}
        \MoveEqLeft
        \norm{R^{h}_{t}-R^{h,N}_{t}}_{L^{2}}^{p}
        \\&\lesssim   \int _{0}^{t} \norm{R^{h}_s-R^{h,N}_s}_{L^{2}}^{p} ds
        \\&\qquad +(1+\norm{X^{h}}_{C_{T}L^{\infty}}^{2\tilde{m}}+\norm{X^{h,N}}_{C_{T}L^{\infty}}^{2\tilde{m}}) \int _{0}^{t} \norm{R^{h}_s-R^{h,N}_s}_{L^{2}}^{p-2}\norm{O_{t}-O^{N}_{t}}_{L^{2}}^2 ds 
        \\&\qquad + (1+\norm{X^{h}}_{C_{T}L^{\infty}}^{2\tilde{m}}+\norm{X^{h,N}}_{C_{T}L^{\infty}}^{2\tilde{m}}) \int _{0}^{t} \norm{R^{h}_s-R^{h,N}_s}_{L^{2}}^{p-1}\norm{O_{t}-O^{N}_{t}}_{L^{2}} ds
        \\&\qquad +  N^{-p\alpha}\norm{g_{h}(\tilde X^{h,N})}_{C_T H^{\alpha}}^p.
    \end{align*}
    Applying again twice Young's inequality for $p_{1}=p/(p-1)$ and $p_{2}=p$ and for $p_{1}=p/(p-2)$ and $p_{2}=p/2$ yields
     \begin{align*}
        \MoveEqLeft
        \norm{R^{h}_{t}-R^{h,N}_{t}}_{L^{2}}^{p}
        \\&\lesssim    \int _{0}^{t} \norm{R^{h}_s-R^{h,N}_s}_{L^{2}}^{p} ds
        \\&\qquad +(1+\norm{X^{h}}_{C_{T}L^{\infty}}^{2\tilde{m}}+\norm{X^{h,N}}_{C_{T}L^{\infty}}^{2\tilde{m}})^p \norm{O-O^{N}}_{C_{T}L^{\infty}}^p
        \\&\qquad +  N^{-p\alpha}\norm{g_{h}(\tilde X^{h,N})}_{C_T H^{\alpha}}^p.
    \end{align*}     
    We choose $\alpha\in (0,1/2)$ and $\epsilon'\in (0,1/2-\alpha)$ and use the embedding $\mathcal{C}^{\beta}\subset H^{\beta'}$ for $\beta>\beta'$ and the bound for $g'_h$ from \cref{lem:g_h-bounds} to obtain that
    \begin{align*}
        \norm{g_{h}( X^{h,N})}_{C_T H^{\alpha}}\lesssim\norm{g_{h}( X^{h,N})}_{C_T \calC^{\alpha+\epsilon'}}\lesssim 
     (1+\norm{X^{h,N}}_{C_{T}L^{\infty}}^{2\tilde m})\norm{X^{h,N}}_{C_{T}\calC^{\alpha+\epsilon'}}.
    \end{align*} 
    Hence Gr\"onwall's inequality together with \eqref{eq:ou-bound} and \eqref{eq:triangle} yield
    \begin{align*}
        \norm{X^{h}_{t}-X^{h,N}_{t}}_{L^{2}}^{p}
        &\lesssim N^{-p(1/2-\epsilon)}\norm{u_{0}}_{\calC^{1/2}}
        +N^{-p\alpha}\norm{O}_{C_{T}\calC^{\alpha+\epsilon'}}^p(1+\norm{X^{h}}_{C_{T}L^{\infty}}^{2\tilde{m}}
        +\norm{X^{h,N}}_{C_{T}L^{\infty}}^{2\tilde{m}})^p
        \\&\qquad +N^{-p\alpha}(1+\norm{X^{h,N}}_{C_T L^{\infty}}^{2\tilde m})^p\norm{X^{h,N}}_{C_{T}\calC^{\alpha+\epsilon'}}^p.
    \end{align*}
    The claim thus follows from 
    the moment bounds from \cref{cor:cont-sol-bounds} and the arbitrariness of $\alpha\in(0,1/2)$.

\end{proof} 
 \begin{lemma}\label{lem:a}
Let \cref{ass:f} hold and let $p\geq 1$. Then there exists a constant $C=C(T,K,m,p,\cM)$ such that
\begin{align*}
    \E\sup_{t\in[0,T]}\norm{u_t-X^{h}_t}_{L^{2}}^{p}\leq C M^{-p}.
\end{align*} 
\end{lemma}
\begin{proof}
We use \cref{rem:weak} with $F=u-X^h$, $G=f(u)-g_h(X^h)$, and $H=0$.
Using the energy estimate and the bounds on $g_h$ and $g_h-g_0$ from \cref{lem:g_h-bounds},
we obtain for $p\geq 2$, $t\in[0,T]$,
\begin{align*}
    \norm{u_t-X^{h}_t}_{L^{2}}^{p}&\leq
    p\int_{0}^{t}\norm{u_s-X^{h}_s}_{L^{2}}^{p-2}\langle  u_s-X^{h}_s, g_{h}(u_s)-g_{h}(X^{h}_s)\rangle ds
    \\&\qquad +p\int_{0}^{t}\norm{u_s-X^{h}_s}_{L^{2}}^{p-2}\langle  u_s-X^{h}_s, g_{0}(u_s)-g_{h}(u_s)\rangle ds
    \\&\lesssim  \int_{0}^{t}\norm{u_s-X^{h}_s}_{L^{2}}^pds + h (1+\norm{u}_{C_{T}L^{\infty}}^{4m+2})\int_{0}^{t}\norm{u_s-X^{h}_s}_{L^{2}}^{p-1} ds.
\end{align*}
The claim thus follows from applying Young's inequality with $p_{1}=p$ and $p_{2}=p/(p-1)$ to the second term, using Gr\"onwall's inequality, taking expectations, and using the a priori bounds from \cref{prop:wp}.  
\end{proof}
\begin{lemma}\label{lem:c}
   Let \cref{ass:f} hold.
    Then for any $p\geq 1$, $\epsilon\in (0,1/4)$, there exists a constant $C=C(T,\epsilon, p,  K, \cM)$  such that the following bound holds
    \begin{align*}
        \E\sup_{t\in[0,T]}\norm{X^{h,N}_{t}-X^{M,N}_{t}}_{L^{2}}^{p}\leq C M^{p(-1+\epsilon)}.
    \end{align*}
\end{lemma}
 To prove the lemma, a key estimate is the following analogue of \cref{prop:main-prop}.
\begin{proposition}\label{prop:crucial}
    Let \cref{ass:f} hold.
    Then for any $p\geq 1$, $\epsilon\in (0,1/4)$, there exists a constant $C=C(T,\epsilon, p,  K, \cM)$  such that the following bound holds
   \begin{align*}
       \E\sup_{t\in[0,T]}\norm[\bigg]{\int_{0}^{t}P^{N}_{t-s}[g_{h}(X^{M,N}_{s})-g_{h}(X^{M,N}_{k_{M}(s)})]ds}_{L^{\infty}}^{p}\leq C M^{p(-1+\epsilon)}.
   \end{align*}
\end{proposition}
 We first give the proof of \cref{lem:c} using \cref{prop:crucial}.
\begin{proof}[Proof of \cref{lem:c}]
Writing $X^{h,N}=R^{h,N}+O^N+P^{N}u_0$ and $X^{M,N}=R^{M,N}+O^N+P^{N}u_0$, we have that
\begin{align}
    X^{h,N}_{t}- X^{M,N}_{t}= R^{h,N}_t-R^{M,N}_t &=\int_{0}^{t}P^{N}_{t-s}g_{h}(X^{h,N}_s)ds-\int_{0}^{t}P^{N}_{t-k_{M}(s)}g_{h}(X^{M,N}_{k_{M}(s)})ds\label{eq:bigbound}
    \\&=\int_{0}^{t}P^{N}_{t-s}[g_{h}(X^{h,N}_s)-g_{h}(X^{M,N}_{s})]ds\label{eq:buckeling}\\&\quad+\int_{0}^{t}P^{N}_{t-s}[g_{h}(X^{M,N}_{s})-g_{h}(X^{M,N}_{k_{M}(s)})]ds\label{eq:sewing} \\&\quad - \int_{0}^{t}[P^{N}_{t-k_{M}(s)}-P_{t-s}^{N}]g_{h}(X^{M,N}_{k_{M}(s)})ds.\label{eq:semigroup}
\end{align}
The last term \eqref{eq:semigroup} is easy to bound with the semigroup estimates, using also the uniform growth bound on $g_h$ from \cref{lem:g_h-bounds}:
\begin{align}
\MoveEqLeft
    \norm[\bigg]{\int_{0}^{t}[P^{N}_{t-k_{M}(s)}-P_{t-s}^{N}]g_{h}(X^{M,N}_{k_{M}(s)})ds}_{L^{\infty}}\nonumber
    \\&=\norm[\bigg]{\int_{0}^{t}P^{N}_{t-s}(P_{s-k_{M}(s)}^{N}-\operatorname{Id})g_{h}(X^{M,N}_{k_{M}(s)})ds}_{L^{\infty}}\nonumber
    \\&\lesssim \int_{0}^{t}(s-k_{M}(s))^{1-\epsilon}(t-s)^{-1+\epsilon}\norm{g_{h}(X^{M,N}_{k_{M}(s)})}_{L^{\infty}}d\nonumber s\\&\lesssim M^{-1+\epsilon}(1+\norm{X^{M,N}}_{C_{T}L^{\infty}}^{2\tilde{m}+1}).\label{eq:R-easy}
\end{align}
The first term \eqref{eq:buckeling} is a buckling term.
However, due to the growth of the Lipschitz constant of $g_h$, we cannot apply Gr\"onwall's inequality in the mild form and we switch again to the variational formulation.
Define
\begin{align*}
    \bar R^{M,N}_{t}=\int_{0}^{t}P^{N}_{t-s}g_{h}(X^{M,N}_{s})ds
\end{align*}
and use \cref{rem:weak} with $F=R^{h,N}-\bar R^{M,N}$, $G=g_h(X^{h,N})-g_h(X^{M,N})$, $H=0$.
From the energy inequality and adding and subtracting appropriate terms, as well as $R^{h,N}=\Pi_{N}R^{h,N}$ and $R^{M,N}=\Pi_{N}R^{M,N}$, we get for any $p\geq 2$
\begin{align}
\MoveEqLeft
    \norm{R^{h,N}_t-\bar R^{M,N}_{t}}_{L^{2}}^{p} 
    \nonumber\\&\leq p\int_{0}^{t}\norm{R^{h,N}_s-\bar R^{M,N}_{s}}_{L^{2}}^{p-2}\langle g_h(X^{h,N}_s)-g_h( X^{M,N}_{s}),R^{h,N}_s- R^{M,N}_{s}\rangle ds
   \nonumber \\&\quad + p\int_{0}^{t}\norm{R^{h,N}_s-\bar R^{M,N}_{s}}_{L^{2}}^{p-2}\langle g_h(X^{h,N}_s)-g_h( X^{M,N}_{s}),R^{M,N}_s- \bar R^{M,N}_{s}\rangle ds.
    \end{align}
    Keeping in mind that $X^{h,N}- X^{M,N}=R^{h,N}- R^{M,N}$, from the bounds for $g_h$ from \cref{lem:g_h-bounds} one therefore has
    \begin{align*}
    \MoveEqLeft
        \norm{R^{h,N}_t-\bar R^{M,N}_{t}}_{L^{2}}^{p}
    \\&\lesssim \int_{0}^{t}\norm{R^{h,N}_s-\bar R^{M,N}_{s}}_{L^{2}}^{p-2}\norm{R^{h,N}_s- R^{M,N}_{s}}_{L^{2}}^2 ds
    \\&\quad +  (1+\norm{X^{h,N}}_{C_{T}L^{\infty}}^{2\tilde{m}}+\norm{X^{M,N}}_{C_{T}L^{\infty}}^{2\tilde{m}})\nonumber\\&\hspace{2cm}\times\int_{0}^{t}\norm{R^{h,N}_s-\bar R^{M,N}_{s}}_{L^{2}}^{p-2}\norm{R^{h,N}_s- R^{M,N}_{s}}_{L^{2}}\norm{R^{M,N}_{s}-\bar R^{M,N}_{s}}_{L^2} ds.
\end{align*}
Utilising Young's inequality to separate all terms, we get
\begin{align*}
    \norm{R^{h,N}_t-\bar R^{M,N}_{t}}_{L^{2}}^{p}\nonumber
    &\lesssim \int_{0}^{t}\norm{R^{h,N}_s-\bar R^{M,N}_{s}}_{L^{2}}^{p}ds+ \int_{0}^{t}\norm{R^{h,N}_s- R^{M,N}_{s}}_{L^{2}}^{p}ds \nonumber
    \\&\quad + [1+\norm{X^{M,N}}_{C_{T}L^{\infty}}^{2\tilde{m}}+\norm{X^{h,M}}_{C_{T}L^{\infty}}^{2\tilde{m}}]^p \int_{0}^{t}\norm{R^{M,N}_s- \bar R^{M,N}_{s}}_{L^{\infty}}^p ds.
\end{align*}
Applying Grönwall's lemma thus yields that
\begin{align}
\MoveEqLeft
    \norm{R^{h,N}_t-\bar R^{M,N}_{t}}_{L^{2}}^{p}\nonumber
    \\&\lesssim  \int_{0}^{t}\norm{R^{h,N}_s- R^{M,N}_{s}}_{L^{2}}^{p}ds\nonumber
    \\&\qquad+[1+\norm{X^{M,N}}_{C_{T}L^{\infty}}^{2\tilde{m}}+\norm{X^{h,N}}_{C_{T}L^{\infty}}^{2\tilde{m}}]^p \int_0^t\norm{R^{M,N}_s- \bar R^{M,N}_s}_{L^{\infty}}^p\,ds. \nonumber
\end{align} 
Let us abbreviate by $K_*$ random variables such that all of their moments are bounded by the a priori bounds \cref{cor:cont-sol-bounds} and \cref{cor:apriori}.
They may change from line to line.
So for example the above inequality can be written as
\begin{align}\label{eq:barr bound}
   \squeeze[1]{\norm{R^{h,N}_t-\bar R^{M,N}_{t}}_{L^{2}}^{p}\lesssim \int_{0}^{t}\norm{R^{h,N}_s- R^{M,N}_{s}}_{L^{2}}^{p}ds+K_*\int_0^t\norm{R^{M,N}_s- \bar R^{M,N}_s}_{L^{\infty}}^p\,ds.}
\end{align}
Notice that the error $ R^{M,N}_t-   \bar R^{M,N}_t$ can be decomposed into the term in \eqref{eq:sewing} and the term in \eqref{eq:semigroup}. 
Thus, since \eqref{eq:sewing} is precisely the term that is bounded in \cref{prop:crucial} and \eqref{eq:semigroup} is bounded in \eqref{eq:R-easy}, we have that
\begin{align}\label{eq:ss-claim}
\E(E_*)^q:=\E\sup_{t\in[0,T]}\norm{R^{M,N}_t- \bar R^{M,N}_{t}}_{L^{\infty}}^q\lesssim M^{q(-1+\epsilon/2)}
\end{align} 
for any $q\geq 1$.
Putting everything together: using \eqref{eq:R-easy} and \eqref{eq:barr bound} in \eqref{eq:bigbound} and using \eqref{eq:ss-claim} we get
\begin{align*}
    \|R^{h,N}_t-R^{M,N}_t\|_{L^2}^p\lesssim K_*\big(M^{p(-1+\eps)} +E_*^p\big)+\int_{0}^{t}\norm{R^{h,N}_s- R^{M,N}_{s}}_{L^{2}}^{p}ds.
\end{align*}
By Gr\"onwall's inequality we get
\begin{align*}
   \sup_{t\in[0,T]} \|R^{h,N}_t-R^{M,N}_t\|_{L^2}^p\lesssim K_*\big(M^{p(-1+\eps)} +E_*^p\big),
\end{align*}
and taking expectations and applying H\"older's inequality to handle $\E(K_\ast E_\ast^p)$, we finally get
\begin{align*}
    \E\sup_{t\in[0,T]}\|R^{h,N}_t-R^{M,N}_t\|_{L^2}^p\lesssim M^{p(-1+\eps)} 
\end{align*}
as claimed.
\end{proof}
\begin{proof}[Proof of \cref{prop:crucial}]
The proof of the proposition relies on stochastic sewing in the form of \cref{lem:ss}, similarly to \cref{prop:main-prop}, however, starting with a somewhat more complicated germ $A_{s,t}$. 
Nevertheless, some of the steps remain unchanged from the proof of \cref{prop:main-prop}, in these cases we shall simply refer back.

Again due to an application of \cref{prop:vKolmogorov} and Besov embeddings, it suffices to prove that for any $j\geq -1$, $x\in\T$, $0\leq s\leq t\leq R\leq T$, one has
 \begin{align}\label{eq:goal}
 \MoveEqLeft
       \E\Big|\int_{s}^{t}\Delta_j P^{N}_{R-r}[g_{h}(R^{M,N}_{r}+O_r^N+P^{N}_{r}u_0)-g_{h}(R^{M,N}_{k_{M}(r)}+O_{k_{M}(r)}^N+P^{N}_{k_{M}(r)}u_0)]dr\Big|^{p}\nonumber
       \\&\leq C 2^{-jp\epsilon}M^{p(-1+2\epsilon)}\abs{t-s}^{(1/4+\epsilon/2)p}.
\end{align}
 To prove \eqref{eq:goal}, we consider $j\geq-1$, $x\in\T$, $R\leq T$ fixed and apply Lemma \ref{lem:ss} with the germ
\begin{align*}
    A_{st}=\int_{s}^{t}\E_{s}[\Delta_j P^{N}_{R-r}(g_{h}(\E_{s} R^{M,N}_{r}+\tilde O_r^N)-g_{h}(\E_{s}R^{M,N}_{k_{M}(r)}+\tilde O_{k_{M}(r)}^N))](x)dr,
\end{align*}
where we denote again by $\tilde{O}_{t}=O_t+P^{N}_{t}u_0$ the shifted OU process.
We aim to prove that for all $0\leq s<u<t<R$ such that $|t-s|\leq |R-t|$, one has
\begin{align}\label{eq:claim1}
        \|A_{ut}\|_{L^p(\Omega)}\lesssim 2^{-j\epsilon} M^{-1+2\epsilon} (R-t)^{-1/4-\epsilon/2}(t-u)^{1/2+\epsilon}
\end{align} and
 \begin{align}\label{eq:claim2}
        \|\E_{s}\delta A_{s,u,t}\|_{L^p(\Omega)}\lesssim 2^{-j\epsilon} M^{-1+2\epsilon} (R-t)^{-1/4-\epsilon/2}(t-s)^{1+\epsilon}.
    \end{align}     
To mimic the steps of the proof of \cref{prop:main-prop},
we recall some analogues of the basic bounds used therein.
Let us now use the shorthand $\cR_{u,r}=\E_uR^{M,N}_r$.

Notice that due to \eqref{eq:u0-bound1} and \eqref{eq:u0-bound2}, we have that $\tilde O$ satisfies the same bounds as $O$, that is \eqref{eq:ou-reg} and \eqref{eq:s-t-bound} for $O$ replaced by $\tilde O$. 

By Jensen's inequality and \cref{cor:apriori} one has for any $q\geq 1$
\begin{align}\label{eq:useful2}
    \big(\E\norm{\cR_{s,r}-\cR_{s,k_M(r)}}_{L^{\infty}}^{q}\big)^{1/q}
    &\leq  \big(\E\norm{R^{M,N}_{r}-R^{M,N}_{k_{M}(r)}}_{L^\infty}^{q}\big)^{1/q}\lesssim_q (r-k_{M}(r))^{1-\epsilon}.
\end{align}
This yields the analogue of
\eqref{eq:ou-reg}
to bound the difference of the two arguments of the nonlinearity:
 \begin{align}\label{eq:R-reg}
 \MoveEqLeft
     \big(\E\norm{\cR_{u,r}+\tilde O_r^N-(\cR_{u,k_{M}(r)}+\tilde O^{N}_{k_M(r)})}_{\calC^{-1/2+\eps}}^q\big)^{1/q}\lesssim_q (r-k_{M}(r))^{1/2-\epsilon}.
 \end{align}
Similarly, one has the bound, if $k_M(r)>u$,
\begin{align}\label{eq:trade-off}
\MoveEqLeft
    \big(\E\norm{\cR_{u,r}+P_{r-u}\tilde O^{N}_u-(\cR_{u,k_{M}(r)}+P_{k_{M}(r)-u}\tilde O^{N}_u)}_{\calC^{-1/2+\eps}}^q\big)^{1/q}\nonumber\\
    &\lesssim_q (r-k_{M}(r))^{1-2\epsilon}(k_{M}(r)-u)^{-1/2+\epsilon},
\end{align}
where we employed \eqref{eq:useful2} and \eqref{eq:s-t-bound} with $\theta=1-4\epsilon$, $r-k_M(r)$ in place of $s$, and $k_M(r)-u$ in place of $t$.
The bound
\eqref{eq:comp-b} 
is replaced by
\begin{align}\label{eq:gh-bound}
\MoveEqLeft
   \E \sup_{r\in[0,R]}\norm{g_{h}'(\lambda (\cR_{u,r}+\tilde O^{N}_r)+(1-\lambda)(\cR_{u,k_{M}(r)}+\tilde O^{N}_{k_{M}(r)}))}_{\calC^{\theta}}^q\nonumber
   \\&\lesssim \E\big((1+\norm{X^{M,N}}_{C_{T}L^{\infty}}^{2\tilde m -1})\norm{X^{M,N}}_{C_{T}\calC^{\theta}}\big)^q\lesssim_{q,\theta} 1.
\end{align}
for $\theta\in (0,1/2)$, which follows using the identification of the Besov space $\calC^{\theta}$ with the space of $\theta$-Hölder continuous functions and the bound for $g_{h}''$ from \cref{lem:g_h-bounds}.

To prove \eqref{eq:claim1}, we distinguish $\abs{t-u}\leq 3TM^{-1}$ and $\abs{t-u}> 3TM^{-1}$ as before.
In the case of $\abs{t-u}\leq 3TM^{-1}$, we estimate similar as in \eqref{eq:long-smallt-s} replacing $ \tilde O_\cdot$  by $\cR_{u,\cdot}+ \tilde O_{\cdot}^{N}$ and use \eqref{eq:R-reg} and \eqref{eq:gh-bound} to obtain the desired estimate.

In the case of $\abs{t-u}> 3TM^{-1}$, we introduce $t'=k_{M}(u)+3TM^{-1}$ as before and decompose the time integral $\int_{u}^{t}$ into $\int_{u}^{t'}$ and $\int_{t'}^{t}$, where $\int_{u}^{t'}$ is treated as in the case before and $\int_{t'}^{t}$ is left to bound.
Recall that for $r\geq t'$ we have $k_M(r)-u\geq(r-u)/2$ and that
$\cR_{u,r}$ and $\cR_{u,k_M(r)}$ are $\F_u$-measurable.
Thus we have, similarly to \eqref{eq:conditional-easy} that the time integral from $t'$ to $t$ equals
\begin{align}
\MoveEqLeft
    \int_{t'}^{t}\Delta_j P_{R-u}\E_{u}[(g_{h}(\cR_{u,r}+\tilde O^{N}_r)-g_{h}(\cR_{u,k_{M}(r)}+\tilde O^{N}_{k_{M}(r)}))]dr \nonumber \\&=
    \int_{u}^{t}\Delta_j P_{R-u}(P^{\R}_{Q^{N}(r-u)}g_h)(\cR_{u,r}+P_{r-u}\tilde O^{N}_u)\nonumber
    \\&\hspace{5cm}-(P^{\R}_{Q^{N}(k_{M}(r)-u)}g_h)(\cR_{u,k_{M}(r)}+P_{k_{M}(r)-u}\tilde O^{N}_u)dr \nonumber
    \\&=\int_{t'}^{t}\Delta_j P_{R-r}^{N}[(P^{\R}_{Q^{N}(r-u)}g_h)(\cR_{u,r}+P_{r-u}\tilde O^{N}_u)\nonumber
    \\&\hspace{5cm}-(P^{\R}_{Q^{N}(r-u)}g_h)(\cR_{u,k_{M}(r)}+P_{k_{M}(r)-u}\tilde O^{N}_u)](x)dr\nonumber
        \\&\quad+\squeeze[0.5]{\int_{t'}^{t}\Delta_j P_{R-r}^{N}[[(P^{\R}_{Q^{N}(r-u)}g_h)-(P^{\R}_{Q^{N}(k_{M}(r)-u)}g_h)](\cR_{u,k_{M}(r)}+P_{k_{M}(r)-u}\tilde O^{N}_u)](x)dr},
        \nonumber
        \\
        &=:I_1+I_2,\nonumber
\end{align}
where
\begin{align*}
Q^{N}(t-s)=\E\paren[\bigg]{\int_{s}^{t}\int_\T p^{N}_{t-r}(x-y)\xi(dr,dy)}^2.
\end{align*} 
We are left to bound the two terms $I_1$ and $I_2$.
To bound the term $I_1$ we proceed just as in \eqref{eq:I1-easy}, but using \eqref{eq:trade-off} instead of \eqref{eq:s-t-bound} and the bound
\begin{align*}
\MoveEqLeft
   \E\norm{(P^{\R}_{Q^{N}(r-u)}g_{h})'(\lambda (P_{r-u}\tilde O^{N}_{u}+\cR_{u,r})+(1-\lambda)(P_{k_{M}(r)-u}\tilde O^{N}_u + \cR_{u,k_{M}(r)}))}_{\calC^{\theta}}^q\lesssim_{\theta,q}1
   \nonumber
\end{align*}
for $\theta\in (0,1/2)$, $q\geq 1$,
in place of \eqref{eq:comp-b2}.
This gives 
\begin{equ}
(\E|I_1|^p)^{1/p}\lesssim     2^{-j\epsilon}M^{-1+2\epsilon} (R-t)^{-1/4-\epsilon/2}(t-u)^{1/2+\epsilon}.
\end{equ}
To bound $I_2$ we replace the bound \eqref{eq:Q-bound} by the following:
\begin{align*}
        \E&\abs[\bigg]{\int_{t'}^{t}\Delta_j P_{R-r}^{N}[[(P^{\R}_{Q^{N}(r-u)}g_h)-(P^{\R}_{Q^{N}(k_{M}(r)-u)}g_h)](\cR_{u,k_{M}(r)}+P_{k_{M}(r)-u}\tilde O^{N}_{u})](x)dr}^p
        \\&\lesssim 2^{-j\epsilon p}\E \paren[\bigg]{\int_{t'}^{t}(R-r)^{-\epsilon/2}
        \\&\quad\qquad\qquad\times\squeeze[0.5]{ \norm{[(P^{\R}_{Q^{N}(r-u)}g_h)-(P^{\R}_{Q^{N}(k_{M}(r)-u)}g_h)](\cR_{u,k_{M}(r)}+P_{k_{M}(r)-u}\tilde O^{N}_{u})}_{L^{\infty}}dr}}^p
        \\&\lesssim 2^{-j\epsilon p}\E \paren[\bigg]{\int_{t'}^{t}(R-r)^{-\epsilon/2} \norm{(P^{\R}_{Q^{N}(r-u)}g_h)-(P^{\R}_{Q^{N}(k_{M}(r)-u)}g_h)}_{C^{0}_{\omega}}\\&\hspace{6.5cm}\times\norm{\omega^{-1}(\cR_{u,k_{M}(r)}+P_{k_{M}(r)-u}\tilde O^{N}_{u})}_{L^{\infty}}dr}^p .
\end{align*}
Then we further estimate using \eqref{eq:HK-3} and \eqref{eq:Q-est} (which also holds true for $Q^N$ instead of $Q$),
\begin{align*}
   \norm{(P^{\R}_{Q^{N}(r-u)}g_h)-(P^{\R}_{Q^{N}(k_{M}(r)-u)}g_h)}_{C^{0}_{\omega}}&\lesssim (Q^{N}(r-u)-Q^{N}(k_{M}(r)-u))\norm{g_h}_{C^{2}_{\omega}} \\&\lesssim (k_{M}(r)-u)^{-1/2+\epsilon}(r-k_{M}(r))^{1-\epsilon}\norm{g_h}_{C^{2}_{\omega}}
\end{align*} 
and
\begin{align*}
   \norm{\omega^{-1}(\cR_{u,k_{M}(r)}+P_{k_{M}(r)-u}\tilde O^{N}_{u})}_{L^{\infty}}\lesssim (1+ \norm{\tilde O^{N}}_{C_{T}L^{\infty}}^2+\norm{R^{M,N}}_{C_{T}L^{\infty}}^2)^{\beta/2}.
\end{align*}
Together with the a priori bounds from \cref{cor:apriori}, we get 
\begin{equ}
(\E|I_2|^p)^{1/p}\lesssim     2^{-j\epsilon}M^{-1+\epsilon} (R-t)^{-\epsilon/2}(t-u)^{1/2+\epsilon}.
\end{equ}
The proof of \eqref{eq:claim1} is therefore finished.

It remains to prove \eqref{eq:claim2}.
Before proceeding, we derive some bounds on $\cR_{\cdot,\cdot}$.
Recall that for any norm $|\cdot|$, any $q\geq 1$, and any random variables $X,Y$, where $Y$ is $\mathcal{F}_s$-measurable, one has $(\E|X-\E_s X|^q)^{1/q}\leq 2(\E|X-Y|^q)^{1/q}$. 
Therefore we can estimate, using \cref{cor:apriori},
\begin{align}\label{eq:useful1}
    \big(\E\norm{R^{M,N}_{r}-\cR_{s,r}}_{L^{\infty}}^{q}\big)^{1/q}&\leq 2 \big(\E\norm{R^{M,N}_{r}-R^{M,N}_{s}}_{L^\infty}^{q}\big)^{1/q}
   \lesssim_q \abs{r-s}^{1-\epsilon}.
\end{align} 
As a consequence, by the triangle inequality, for $s\leq u$
\begin{align}\label{eq:easier-bound}
\MoveEqLeft
(\E\norm{\cR_{s,r}-\cR_{u,r}}_{L^{\infty}}^{q})^{1/q}
\lesssim_q  |r-s|^{1-\epsilon}.
\end{align}
Similarly, we can bound for $s\leq u$
\begin{align}\label{eq:easier-bound2}
    (\E\norm{\cR_{s,r}-\cR_{u,r}}_{\calC^{1/2-\epsilon/2}}^{q})^{1/q}\lesssim_q \abs{r-s}^{3(1-\epsilon)/4},
\end{align} using again the a priori estimates, \cref{cor:apriori} with $\alpha=1/2-\epsilon$.
Next, we claim that\footnote{Notice that simply substituting $k_M(r)$ in place of $r$ in \eqref{eq:easier-bound} does not give the required order, since one may have $|k_M(r)-s|\gg |r-s|$.}
\begin{align}\label{eq:not-easier-bound1}
\E\big(\norm{R^{M,N}_{k_{M}(r)}-\cR_{s,k_{M}(r)}}_{L^{\infty}}^{q}\big)^{1/q}\lesssim_q \abs{r-s}^{1-\epsilon}.
\end{align}
Indeed, to see this, we make the following case separation as in \cite[Proof of Lemma 4.7]{BDG-Levy}: If $s\geq k_{M}(r)-TM^{-1}$ then $\cR_{s,k_{M}(r)}=\E_{s}R^{M,N}_{k_{M}(r)}=R^{M,N}_{k_{M}(r)}$ and the left-hand side vanishes, such that the bound is trivial. If $s\leq k_{M}(r)-TM^{-1}$, then  $|k_M(r)-s|\leq |r-s|$ and the bound follows from \eqref{eq:useful1} plugging $k_{M}(r)$ in place of $r$.\\
Moreover, using triangle inequality and \eqref{eq:not-easier-bound1} we obtain for $s\leq u$,
\begin{align}\label{eq:not-easier-bound}
\MoveEqLeft
(\E\norm{\cR_{s,k_M(r)}-\cR_{u,k_M(r)}}_{L^{\infty}}^{q})^{1/q}
\lesssim_q  |r-s|^{1-\epsilon}.
\end{align}
Now we move on to prove \eqref{eq:claim2},
for which we have to bound the term
\begin{align*}
    \E_{s}\delta A_{sut}=\Delta_j \int_{u}^{t}\E_s\E_u P^{N}_{R-r}\paren[\big]{&g_{h}(\cR_{s,r}+\tilde O^{N}_r)-g_{h}(\cR_{s,k_{M}(r)}+\tilde O^{N}_{k_{M}(r)})
    \\&-(g_{h}(\cR_{u,r}+\tilde O^{N}_r)-g_{h}(\cR_{u,k_{M}(r)}+\tilde O^{N}_{k_{M}(r)}))}(x)dr.
\end{align*}
Let us write 
\begin{align*}
\MoveEqLeft
    g_{h}(\cR_{s,r}+\tilde O^{N}_r)-g_{h}(\cR_{s,k_{M}(r)}+\tilde O^{N}_{k_{M}(r)})
    -(g_{h}(\cR_{u,r}+\tilde O^{N}_r)-g_{h}(\cR_{u,k_{M}(r)}+\tilde O^{N}_{k_{M}(r)})) \\&= (\cR_{s,r}-\cR_{u,r})\int_{0}^{1}g'_{h}(\lambda (\cR_{s,r}+\tilde O^{N}_{r})+(1-\lambda)(\cR_{u,r}+\tilde O^{N}_r))d\lambda
    \\&\quad +(\cR_{s,k_{M}(r)}-\cR_{u,k_{M}(r)})
    \\&\hspace{3cm}\times\squeeze[0.5]{\int_{0}^{1}g'_{h}(\lambda (\cR_{s,k_{M}(r)}+\tilde O^{N}_{k_{M}(r)})+(1-\lambda)(\cR_{u,k_{M}(r)}+\tilde O^{N}_{k_{M}(r)}))d\lambda}
    \\&= (\cR_{s,r}-\cR_{u,r})\int_{0}^{1}\paren[\bigg]{g'_{h}(\lambda (\cR_{s,r}+\tilde O^{N}_{r})+(1-\lambda)(\cR_{u,r}+\tilde O^{N}_r))
    \\&\hspace{3.5cm}-\squeeze[0.5]{g'_{h}(\lambda (\cR_{s,k_{M}(r)}+\tilde O^{N}_{k_{M}(r)})+(1-\lambda) (\cR_{u,k_{M}(r)}+\tilde O^{N}_{k_{M}(r)}))}d\lambda}
    \\&\quad + (\cR_{s,k_{M}(r)}-\cR_{s,r}+\cR_{u,r}-\cR_{u,k_{M}(r)})
    \\&\hspace{3cm}\times\squeeze[0.5]{\int_{0}^{1}g'_{h}(\lambda (\cR_{s,k_{M}(r)}+\tilde O^{N}_{k_{M}(r)})+(1-\lambda)(\cR_{u,k_{M}(r)}+\tilde O^{N}_{k_{M}(r)}))d\lambda}
    \\&=:J_1+J_2.
\end{align*}
By the bounds for $g_{h}$ from \cref{lem:g_h-bounds}, as well as using \eqref{eq:useful2} and \eqref{eq:easier-bound}, \eqref{eq:not-easier-bound}, we obtain for $J_2$ the following estimate
\begin{align*}
\MoveEqLeft
    \paren[\bigg]{\E\abs[\bigg]{\Delta_j \int_{u}^{t}\E_{s}[P^{N}_{R-r}\E_{u}[J_2]](x)dr}^{p}}^{1/p}\\&\lesssim
    2^{-j\epsilon}(R-t)^{-\epsilon/2}(1+\E\norm{R^{M,N}}_{C_{T}L^{\infty}}^{2\tilde{m}p}+\E\norm{\tilde O^{N}}_{C_{T}L^{\infty}}^{2\tilde{m}p})^{1/2p}
    \\&\hspace{3cm}\times\paren[\bigg]{\E\paren[\bigg]{\int_{u}^{t}\norm{\cR_{s,k_{M}(r)}-\cR_{s,r}+\cR_{u,r}-\cR_{u,k_{M}(r)}}_{L^{\infty}}dr}^{2p}}^{1/2p}
    \\&\lesssim 2^{-j\epsilon}(R-t)^{-\epsilon/2}
    \int_{u}^{t}\min((r-k_{M}(r))^{1-\epsilon/2}, (r-s)^{1-\epsilon/2})dr
    \\&\lesssim 2^{-j\epsilon}(R-t)^{-\epsilon/2}
    (t-s)^{1+\epsilon/2}M^{-1+\epsilon},
\end{align*} 
where the last bound follows by the interpolation $\min(a,b)\leq a^{1-\theta}b^{\theta}$ for $\theta=\epsilon/(1-\epsilon)\in (0,1)$, $a,b\geq 0$.

To estimate $J_{1}$ we distinguish again the cases $\abs{t-u}\leq 3TM^{-1}$ and $\abs{t-u}>3TM^{-1}$ like for $A_{ut}$. In the first case, the estimate is carried out as in \eqref{eq:I1-easy} using \eqref{eq:R-reg} in place of \eqref{eq:ou-reg}.  Notice that \eqref{eq:easier-bound2} gives an additional factor $(r-s)^{3(1-\epsilon)/4}$, and thus we obtain that in the case of $\abs{t-u}\leq 3M^{-1}$,
\begin{equ}
\paren[\bigg]{\E\abs[\bigg]{\Delta_j \int_{u}^{t}\E_{s}[P^{N}_{R-r}\E_{u}[J_1]](x)dr}^{p}}^{1/p}\lesssim     2^{-j\epsilon}M^{-1+\epsilon} (R-t)^{-\epsilon/2}(t-u)^{5/4+\epsilon/4}.
\end{equ}
If $\abs{t-u}>3TM^{-1}$, we introduce $t'$ as before, in particular $k_{M}(r)\geq u\geq s$.
The integral from $u$ to $t'$ is treated as in the case $|t-u|\leq 3TM^{-1}$. For the integral from $t'$ to $t$ we have that, using that the factor $\cR_{s,r}-\cR_{u,r}$ is $\F_{u}$-measurable,
\begin{align*}
    &\paren[\bigg]{\E\abs[\bigg]{\Delta_j \int_{t'}^{t}\E_{s}[P^{N}_{R-r}\E_{u}[J_1]](x)dr}^{p}}^{1/p}\\&\lesssim 2^{-j\epsilon}(R-t)^{-1/4-\epsilon/2} \paren[\bigg]{\E\int_{t'}^{t}\norm{\cR_{s,r}-\cR_{u,r}}_{\calC^{1/2-\epsilon/2}}^p\\&\quad\times\int_{0}^{1}\norm[\Big]{(P^{\R}_{Q^{N}(r-u)}g_{h}')(\lambda \cR_{s,r}+(1-\lambda)\cR_{u,r}+P_{r-u}\tilde O^{N}_u)
    \\&\qquad\quad-(P^{\R}_{Q^{N}(k_{M}(r)-u)}g_{h}')(\lambda \cR_{s,k_{M}(r)}+(1-\lambda)\cR_{u,k_{M}(r)}+P_{k_{M}(r)-u}\tilde O^{N}_u)}_{\calC^{-1/2+\epsilon}}^p d\lambda dr}^{1/p}.
\end{align*}
Using similar steps as for estimating the time integral from $t'$ to $t$ above, now applied for $g_{h}'\in C^{2}_{\omega}$ instead of $g_{h}$, to bound
\begin{align*}
\MoveEqLeft
    \paren[\bigg]{\E\int_{t'}^{t}\norm{(P^{\R}_{Q^{N}(r-u)}g_{h}')(\lambda \cR_{s,r}+(1-\lambda)\cR_{u,r}+P_{r-u}\tilde O^{N}_u)
    \\&\hspace{0.6cm}\squeeze[0.5]{-(P^{\R}_{Q^{N}(k_{M}(r)-u)}g_{h}')(\lambda \cR_{s,k_{M}(r)}+(1-\lambda)\cR_{u,k_{M}(r)}+P_{k_{M}(r)-u}\tilde O^{N}_u)}}_{\calC^{-1/2+\epsilon}}^{2p} dr}^{1/2p}
\end{align*}
for any fixed $\lambda\in[0,1]$, as well as \eqref{eq:easier-bound2} and \eqref{eq:trade-off}, we arrive at
\begin{align*}
\MoveEqLeft
    \paren[\bigg]{\E\abs[\bigg]{\Delta_j \int_{t'}^{t}\E_{s}[P^{N}_{R-r}\E_{u}[J_1]](x)dr}^{p}}^{1/p}\\&\lesssim
    2^{-j\epsilon}(t-s)^{3(1-\epsilon)/4}(R-t)^{-1/4-\epsilon/2}
    \int_{t'}^{t}(r-k_{M}(r))^{1-2\epsilon}(k_{M}(r)-u)^{-1/2+\epsilon}dr
    \\&\lesssim 2^{-j\epsilon}(t-s)^{5/4+\epsilon/4}(R-t)^{-1/4-\epsilon/2}M^{-1+2\epsilon} .
\end{align*}
The estimates for $J_1$ and $J_2$ together yield \eqref{eq:claim2}.
Hence an application of \cref{lem:ss} together with \eqref{eq:claim1} and \eqref{eq:claim2}  yields \eqref{eq:goal}, provided we justify that 
\begin{align}\label{eq:identify}
    \mathcal{A}_t=\int_{0}^{t}\Delta_j P^{N}_{R-r}[g_{h}(R^{M,N}_{r}+\tilde O_r^N)-g_{h}(R^{M,N}_{k_{M}(r)}+\tilde O_{k_{M}(r)}^N)](x)dr.
\end{align}
To this end, we need to verify \eqref{eq:ss1} and \eqref{eq:ss2}. We do so in two steps: we show the bounds with $\mathcal{A}_{st}-A_{st}$ replaced by $\mathcal{A}_{st}-\tilde A_{st}$ and by $\tilde A_{st}-A_{st}$,
where the intermediate process $\tilde A_{st}$ is defined by
\begin{align*}
    \tilde{A}_{s,t}=\int_{s}^{t}\E_{s}[\Delta_j P^{N}_{R-r}(g_{h}(R^{M,N}_{r}+\tilde O^{N}_r)-g_{h}(R^{M,N}_{k_{M}(r)}+\tilde O^{N}_{k_{M}(r)}))](x)dr.
\end{align*}
The difference $\mathcal{A}_{st}-\tilde A_{st}$ is treated similarly to the argument concluding the proof of \cref{prop:main-prop}: \eqref{eq:ss2} is satisfied with $K_2=0$, and
\begin{align*}
    \paren[\big]{\E\abs[\big]{\mathcal{A}_{st}-\tilde{A}_{st}}^p}^{1/p}\leq  4(t-s)(1+\E\norm{X^{M,N}}_{C_{T}L^{\infty}}^{p(2\tilde{m}+1)})^{1/p},
\end{align*}
verifying \eqref{eq:ss1}.
The verification of \eqref{eq:ss1} for $\tilde A_{st}-A_{st}$ is identical. As for \eqref{eq:ss2},
we clearly have
\begin{align*}
\MoveEqLeft
    (\E\abs{A_{st}-\tilde{A}_{st}}^p)^{1/p}\\&\lesssim (t-s)(1+\E\norm{X^{M,N}}_{C_{T}L^{\infty}}^{2p(2\tilde{m}+1)})^{1/2p}\
    \\
    &\times \sup_{r\in[s,t]}\Big(\big(\E\norm{R^{M,N}_{r}-\cR_{s,r}}_{L^{\infty}}^{2p}\big)^{1/2p}
    +\big(\E\norm{R^{M,N}_{k_{M}(r)}-\cR_{s,k_{M}(r)}}_{L^{\infty}}^{2p}\big)^{1/2p}\Big)
    \\&\lesssim (t-s)^{2-\epsilon}
\end{align*}
where we used \eqref{eq:useful1} and \eqref{eq:not-easier-bound1}.
This verifies \eqref{eq:ss2} for $\tilde A_{st}-A_{st}$, which proves \eqref{eq:identify} and brings the proof to an end.
\end{proof}

\begin{appendix}
\section*{Appendix}\label{Appendix}
\begin{proof}[Proof of \cref{prop:Gronwall}]
We have that for $0\leq l\leq t\leq T$,
\begin{align*}
        \norm{X_l-Y_l}&\leq\norm{Z_l}+\int_0^l\norm{S(s,l)\big(F(X_{\tau(s)})-F(Y_{\tau(s)})\big)}\,ds
        \\&\leq \norm{Z_l} + L_{2}L_{1}\int_{0}^{l} \norm{X_{\tau(s)}-Y_{\tau(s)}} ds
        \\&\leq \norm{Z_l} + L_{2}L_{1}\int_{0}^{t} \sup_{r\leq s}\norm{X_{r}-Y_{r}} ds,
\end{align*}
where we used that $\tau(s)\leq s$.
Thus we deduce that 
\begin{align*}
\sup_{0\leq l\leq t}\norm{X_l-Y_l}\leq \sup_{0\leq l\leq t}\norm{Z_l}+L_{2}L_{1}\int_{0}^{t} \sup_{0\leq r\leq s}\norm{X_{r}-Y_{r}} ds.
\end{align*}
An application of the classical Grönwall inequality to the increasing map $[0,T]\ni t\mapsto \sup_{0\leq l\leq t}\norm{X_l-Y_l}$ yields that
\begin{align*}
\sup_{t\in[0,T]}\norm{X_t-Y_t}\leq \exp(L_{2}L_{1}T)\sup_{t\in[0,T]}\norm{Z_t},
\end{align*} which yields the claim after taking the $p$-th moment.
\end{proof}

\begin{proof}[Proof of \cref{prop:vKolmogorov}]
The proof follows closely the one of the standard Kolmogorov continuity theorem.
Without loss of generality due to rescaling, we may assume $T=1$.
Note that for $s\leq t$ and a partition $(r_{k})_{k=0,\dots,K}$ of $[s,t]$ with $r_0=s$, $r_{K}=t$, we can write, using the semigroup property, 
\begin{align}\label{eq:decomp}
X_{t}-S_{t-s}X_s= \sum_{k=0}^{K-1} S_{t-r_{k+1}}\paren[\Big]{X_{r_{k+1}}-S_{r_{k+1}-r_{k}}X_{r_{k}}}.
\end{align}
Recall also that for a strongly continuous semigroup there exist constants $C\geq 1$, $c\geq 0$, such that for all $r\geq 0$, $\norm{S_{r}}\leq C e^{c r}$.
Let for $n\in\N$, $D_n$ be the dyadic partition of $[0,1]$ with mesh size $2^{-n}$ and set $D=\cup_{n}D_{n}$. Then due to the assumed continuity of $t\mapsto X_t$ and $t\mapsto S_t X_s$ for any fixed $s\in[0,1]$, it suffices to show that there exists $C''<\infty$, such that
\begin{align}\label{eq:claim}
    \E\big(\sup_{s< t\in D}|t-s|^{-\gamma }\|X_t-S_{t-s}X_s\|\big)^p\leq C''C',
\end{align}
Let for $\gamma\in (0,\alpha/p)$,  
$$K_{\gamma}:=\sup_{n\in\N}\sup_{t,u\in D_{n}, \, t+\frac{1}{2^{n}}\leq u} 2^{\gamma n} \norm{X_{t+2^{-n}}-S_{2^{-n}}X_t}.$$
Using the assumption and $\abs{D_n}-1=2^n$, we obtain that
\begin{align}\label{eq:p-bound}
\E K_{\gamma}^p&\leq \sum_{n\in\N}\sum_{t\in D_n\setminus\{1\}}2^{\gamma n p} \E \norm{X_{t+2^{-n}}-S_{2^{-n}}X_t}^p\nonumber\\&\leq \sum_{n\in\N} (|D_n|-1) 2^{n\gamma p} C' 2^{-n(1+\alpha)} = C' \sum_{n\in\N} 2^{-n (\alpha-\gamma p)}=:C'\tilde C ,
\end{align}
where the sum is finite due to $\gamma p<\alpha$.

Let now $s,t\in D$ with $0\leq s\leq t\leq 1$ (excluding the case $s=0$ and $t=1$, which is covered above) and let $m\in\N$ be the largest integer such that $D_{m}\cap [s,t]$ is a singleton, which we denote by $\{t_0\}$. We have that $2^{-m-1}\leq(t-s)\leq 2^{-m+1}$. Define inductively a sequence $(t_{n})$, such that $t_{n}\in D_{m+n}$, $t_{n}\leq t$ and $(t-t_n)<2^{-(m+n)}$. The case $n=0$ is clear. For the inductive step, set $t_{n+1}=t_{n}+2^{-(m+n+1)}$ if $t_n + 2^{-(m+n+1)}\leq t$ and otherwise set $t_{n+1}=t$. Analogously we define a sequence $(s_n)$ but in a decreasing way. Since $s,t\in \cup_{n} D_n$, there exist $n_s,n_t$ such that $t_{n}=t$ for all $n\geq n_t$ and $s_{n}=s$ for all $n\geq n_s$. By the decomposition \eqref{eq:decomp}, we have
\begin{align*}
X_t - S_{t-t_0} X_{t_0} &= \sum_{i=0}^{n_t-1} S_{t-t_{i+1}}(X_{t_{i+1}}-S_{2^{-(m+i+1)}}X_{t_i}),\\
S_{t-t_0}(X_{t_0} - S_{t_0-s} X_s) &= -\sum_{i=0}^{n_s -1} S_{t-s_{i+1}}(X_{s_{i+1}}-S_{2^{-(m+i+1)}}X_{s_i})
\end{align*}
and subtraction yields
\begin{align}\label{eq:decomp2}
X_t-S_{t-s}X_s=\sum_{i=0}^{n_t} S_{t-t_{i+1}}(X_{t_{i+1}}-S_{2^{-(m+i+1)}}X_{t_i})+\sum_{i=0}^{n_s} S_{t-s_{i+1}}(X_{s_{i+1}}-S_{2^{-(m+i+1)}}X_{s_i}).
\end{align}
Since $t_{i+1}=t_{i}+2^{-(m+n+1)}$ and $s_{i}=s_{i+1}+2^{-(m+n+1)}$, we can bound each of the summands in \eqref{eq:decomp2} by $K_{\gamma}2^{-\gamma (m+i+1)}$. Using triangle inequality and the uniform bound on the operator-norm for $t\in[0,1]$, we thus arrive at
\begin{align*}
\norm{X_t-S_{t-s}X_s}\leq Ce^c \sum_{i=0}^{\infty} K_{\gamma} 2^{-\gamma (m+i)}\leq 2^{-1}Ce^c\abs{t-s}^{\gamma}K_{\gamma}\sum_{i=0}^{\infty}  2^{-\gamma i},
\end{align*}
where in the last estimate we used that $2^{-m}\leq 2\abs{t-s}$ and the sum is finite due to $\gamma>0$. Taking supremum and $p$-th moment, \eqref{eq:claim} thus follows from \eqref{eq:p-bound} with $C'':=\paren[\big]{2^{-1}Ce^c\sum_{i=0}^{\infty}  2^{-\gamma i}}^p \tilde C$.
\end{proof}

\begin{proof}[Proof of \cref{lem:g_h-bounds}]
We have that 
\begin{align*}
    \Phi_{h}(z)&=z+\int_{0}^{h}f(\Phi_{s}(z))ds,\quad h\geq 0
    \\g_{h}(z)&=\frac{1}{h}\int_{0}^{h}f(\Phi_{s}(z))ds,\quad h>0.
\end{align*}
Thus, we obtain
\begin{align*}
    (\Phi_{h}(x)-\Phi_{h}(y))^2&=(x-y)^2+2\int_{0}^{h}(f(\Phi_{s}(x))-f(\Phi_{s}(y)))(\Phi_{s}(x)-\Phi_{s}(y))ds
    \\&\leq (x-y)^2 + 2K\int_{0}^{h}(\Phi_{s}(x)-\Phi_{s}(y))^2ds
\end{align*} and hence with Grönwall's inequality
\begin{align*}
    (\Phi_{h}(x)-\Phi_{h}(y))^2&\leq e^{Kh} (x-y)^2.
\end{align*}
This implies that $z\mapsto \Phi_{h}(z)$ is Lipschitz continuous with Lipschitz constant $L(h)=e^{Kh/2}$, that satisfies $L(h)^{1/h}=e^{K/2}$. 
Since $\Phi_{h}$ is Lipschitz, we moreover obtain that 
$$\abs{\partial_z\Phi_{s}(z)}\leq e^{Ks/2}\leq e^{Kh/2}$$
for $s\in[0,h]$.
Moreover we have that, using Young's inequality 
\begin{align*}
    \Phi_{h}(z)^2&=z^2+2\int_{0}^{h}f(\Phi_{s}(z))\Phi_{s}(z)ds
    \\&=z^{2}+2\int_{0}^{h}(f(\Phi_{s}(z))-f(0))\Phi_{s}(z)ds
    +2\int_{0}^{h}f(0)\Phi_{s}(z)ds
    \\&\leq z^{2}+2K\int_{0}^{h}\Phi_{s}(z)^2ds
    +hf(0)^2+\int_{0}^{h}\Phi_{s}(z)^2ds,
\end{align*} such that Gr\"onwall's inequality implies that
\begin{align*}
    \sup_{s\in[0,h]}\abs{\Phi_{s}(z)}^2\leq e^{(2K+1)h}(h\abs{f(0)}^2+\abs{z}^2)
\end{align*} and thus
\begin{align}\label{eq:phi-growth}
    \sup_{s\in[0,h]}\abs{\Phi_{s}(z)}\leq e^{(2K+1)h/2}(h^{1/2}\abs{f(0)}+\abs{z})\leq Ke^{(2K+1)h/2}(1+\abs{z}),
\end{align} for $h\leq 1$ and using that $\abs{f(0)}\leq K $.
Furthermore, we have that
\begin{align*}
    (\partial_{z}^{2}\Phi_{h}(z))^2&=1+2\int_{0}^{h}(\partial f)(\Phi_{s}(z))(\partial_{z}^2\Phi_{s}(z))^2ds\\&\qquad +2\int_{0}^{h}(\partial^{2}f)(\Phi_{s}(z))(\partial_{z}\Phi_{s}(z))^2  \partial_{z}^2\Phi_{s}(z)ds
\end{align*} and thus, using Young's inequality, the bounds on $f$ and \eqref{eq:phi-growth}, as well as the Lipschitz bound of $\Phi_{s}$, there exist constants $C'(K,m), C(K,m)>1$, such that
\begin{align*}
   (\partial_{z}^{2}\Phi_{h}(z))^2&\leq 1+ 2K\int_{0}^{h}(\partial_{z}^2\Phi_{s}(z))^2ds + \int_{0}^{h}((\partial^{2}f)(\Phi_{s}(z))(\partial_{z}\Phi_{s}(z))^2)^2ds \\&\qquad +  \int_{0}^{h}(\partial_{z}^2\Phi_{s}(z))^2ds
   \\&\leq 1+ (2K+1)\int_{0}^{h}(\partial_{z}^2\Phi_{s}(z))^2ds + h C'(K,m) e^{h C(K,m)}(1+\abs{z}^{2(2m-1)}).
\end{align*} 
Thus Gr\"onwall's inequality implies that for possibly different constants and $h\leq 1$,
\begin{align}\label{eq:phi-i1}
    \sup_{s\in[0,h]}\abs{\partial^{2}_z\Phi_{s}(z)}\leq C'(K,m)e^{C(K,m) h} (1+\abs{z}^{2m-1}).
\end{align} 
The bound for $\partial_{z}^{3}\Phi_{h}(z)$ follows in a similar fashion using \eqref{eq:phi-i1} and
\begin{align*}
    (\partial_{z}^{3}\Phi_{h}(z))^2&=1+2\int_{0}^{h}(\partial f)(\Phi_{s}(z))(\partial_{z}^3\Phi_{s}(z))^2ds\\&\qquad +6\int_{0}^{h}(\partial^{2}f)(\Phi_{s}(z))\partial_{z}\Phi_{s}(z)  \partial_{z}^2\Phi_{s}(z)\partial_{z}^{3}\Phi_{s}(z)ds \\& \qquad+ 2\int_{0}^{h}(\partial^{3}f)(\Phi_{s}(z))(\partial_{z}\Phi_{s}(z))^3 \partial_{z}^3\Phi_{s}(z)ds
\end{align*}
and we obtain, for possibly different constants $C'(K,m),C(K,m)$ and $\tilde{m}\geq m$, $h\leq 1$,
\begin{align}\label{eq:phi-i2}
    \sup_{s\in[0,h]}\abs{\partial^{3}_z\Phi_{s}(z)}\leq  C'(K,m)e^{C(K,m) h} (1+\abs{z}^{2\tilde{m}-3}).
\end{align} 
We have that 
\begin{align*}
     \partial_z g_{h}(z)&=\frac{1}{h}\int_{0}^{h}(\partial f)(\Phi_{s}(z))\partial_z \Phi_{s}(z)ds, 
     \\\partial^{2}_z g_{h}(z)&=\frac{1}{h}\int_{0}^{h}(\partial^{2}f)(\Phi_{s}(z))(\partial_z \Phi_{s}(z))^{2}ds+\frac{1}{h}\int_{0}^{h}(\partial f)(\Phi_{s}(z))\partial_z^2 \Phi_{s}(z))ds, 
     \\\partial^{3}_z g_{h}(z)&=\frac{1}{h}\int_{0}^{h}(\partial^{3}f)(\Phi_{s}(z))(\partial_z \Phi_{s}(z))^{3}ds+
     3\frac{1}{h}\int_{0}^{h}(\partial^{2}f)(\Phi_{s}(z))(\partial_z^2 \Phi_{s}(z))\partial_z \Phi_{s}(z)ds
     \\&\qquad+\frac{1}{h}\int_{0}^{h}(\partial f)(\psi_{s}(z))\partial_z^3 \Phi_{s}(z))ds.
 \end{align*} 
Together with \eqref{eq:phi-growth}, \eqref{eq:phi-i1} and \eqref{eq:phi-i2} we thus obtain 
\begin{align*}
    \abs{g_h(z)}\leq K (1+\sup_{s\in [0,1]}\abs{\Phi_{s}(z)}^{2m+1})\leq K^2e^{(2K+1)h/2}(1+\abs{z}^{2m+1})
\end{align*} with $K^2e^{(2K+1)h/2}\leq K^{2}e^{(2K+1)/2}$ for $h\in[0,1]$ and
\begin{align*}
    \abs{\partial_z g_h (z)}\leq K e^{Kh/2} (1+\sup_{s\in [0,1]}\abs{\Phi_{s}(z)}^{2m})\leq C'(K,m)e^{hC(K,m)}(1+\abs{z}^{2m}).
\end{align*} 
Using \eqref{eq:phi-i1}, \eqref{eq:phi-i2}, we find for the higher order derivatives, $i=2,3$, that there exists constants $\tilde{m}\geq m$, $K(h,m)>1$, such that
\begin{align*}
    \abs{\partial^{i}_z g_{h}(z)}&\leq K(h,m)(1+\abs{z}^{2\tilde m -i}),
\end{align*} where $K(h,m)\leq \tilde K$ for $h\in[0,1]$ for a constant $\tilde K>1$.
Moreover, since $\partial_h \partial_{z}\Phi_{h}(z)=(\partial f)(\Phi_{h}(z))\partial_{z}\Phi_{h}(z)$ by Gr\"onwall's inequality and $\partial f\leq K$, we have that
\begin{align*}
    \partial_z \Phi_{h}(z)\leq e^{\int_{0}^{h}(\partial f)(\Phi_{s}(z))ds}\leq e^{hK}
\end{align*}
and thus
\begin{align*}
    \partial_z g_{h}(z)=\frac{1}{h}(\partial_z \Phi_{h}(z)-1)\leq \frac{1}{h}(e^{hK}-1)\leq K.
\end{align*}
Finally, we can bound using the bounds on $f$ and \eqref{eq:phi-growth}
\begin{align*}
    \abs{g_{h}(z)-g_{0}(z)}=\abs{g_{h}(z)-f(z)}&=\frac{1}{h}\abs[\bigg]{\int_{0}^{h}f(\Phi_{s}(z))-f(z)ds}
    \\&\leq K^{2}e^{Kh/2}(1+\abs{z}^{2m+1})\sup_{s\in[0,h]}\abs{\Phi_{s}(z)-z}
\end{align*}
and
\begin{align*}
    \abs{\Phi_{s}(z)-z}\leq s K^{2}e^{Kh/2} (1+\abs{z}^{2m+1}).
\end{align*}
Together we thus obtain
\begin{align*}
    \abs{g_{h}(z)-g_{0}(z)}\leq h K^{4}e^{Kh}(1+\abs{z}^{2(2m+1)}), \quad h\in[0,1].
\end{align*}
\end{proof}
\end{appendix}

\newcommand{\etalchar}[1]{$^{#1}$}

\end{document}